\documentclass[a4paper,11pt]{article}
\usepackage[latin1]{inputenc}
\usepackage[english]{babel}
\usepackage{amsmath}
\usepackage{amsfonts}
\usepackage{amssymb}
\usepackage{epsfig}
\usepackage{amsopn}
\usepackage{amsthm}
\usepackage{color}
\usepackage{graphicx}
\usepackage{enumerate}
\newtheorem{theorem}{Theorem}[section]
\newtheorem{corollary}[theorem]{Corollary}
\newtheorem{lemma}[theorem]{Lemma}
\newtheorem{proposition}[theorem]{Proposition}
\newtheorem{definition}[theorem]{Definition}
\newtheorem{remark}[theorem]{Remark}
\newtheorem*{theorem*}{Theorem}
\newtheorem*{lemma*}{Lemma}
\newtheorem*{remark*}{Remark}
\newtheorem*{definition*}{Definition}
\newtheorem*{proposition*}{Proposition}
\newtheorem*{corollary*}{Corollary}
\numberwithin{equation}{section}
\newcommand{\real}{\mathbb{R}}

\let\ced=\c         

\def\e{\varepsilon}        

\def\ca{{\cal A}}

\def\cf{{\cal F}}

\def\cp{{\cal P}}

\newcommand{\RR}{\mathbb{R}}

\def\qed{\,\unskip\kern 6pt \penalty 500
\raise -2pt\hbox{\vrule \vbox to8pt{\hrule width 6pt
\vfill\hrule}\vrule}\par}
\definecolor{darkblue}{rgb}{0.05, .05, .65}
\definecolor{darkgreen}{rgb}{0.1, .65, .1}
\definecolor{darkred}{rgb}{0.8,0,0}
\newcommand{\beqn}{\begin{equation}}
\newcommand{\eeqn}{\end{equation}}
\newcommand{\bear}{\begin{eqnarray}}
\newcommand{\eear}{\end{eqnarray}}
\newcommand{\bean}{\begin{eqnarray*}}
\newcommand{\eean}{\end{eqnarray*}}
\begin{document}

\title{\huge \bf  Positivity, decay, and extinction for a singular diffusion equation with gradient absorption}

\author{
\Large Razvan Gabriel Iagar\,\footnote{Institut de Math\'ematiques
de Toulouse, CNRS UMR~5219, Universit\'e de Toulouse, F--31062
Toulouse Cedex 9, France.},\footnote{Institute of Mathematics of the
Romanian Academy, P.O. Box 1-764, RO-014700, Bucharest, Romania,
\textit{e-mail:} razvan.iagar@imar.ro.}
\\[4pt] \Large Philippe Lauren\c cot\,\footnote{Institut de
Math\'ematiques de Toulouse, CNRS UMR~5219, Universit\'e de
Toulouse, F--31062 Toulouse Cedex 9, France. \textit{e-mail:}
Philippe.Laurencot@math.univ-toulouse.fr}\\ [4pt] }
\date{\today}
\maketitle

\begin{abstract}
We study qualitative properties of non-negative solutions to the
Cauchy problem for the fast diffusion equation with gradient
absorption
\begin{equation*}
\partial_t u -\Delta_{p}u+|\nabla u|^{q}=0\quad \mbox{ in }\;\;
(0,\infty)\times\RR^N,
\end{equation*}
where $N\ge 1$, $p\in(1,2)$, and $q>0$. Based on gradient estimates
for the solutions, we classify the behavior of the solutions for
large times, obtaining either positivity as $t\to\infty$ for
$q>p-N/(N+1)$, optimal decay estimates as $t\to\infty$ for $p/2\le
q\le p-N/(N+1)$, or extinction in finite time for $0<q<p/2$. In
addition, we show how the diffusion prevents extinction in finite
time in some ranges of exponents where extinction occurs for the
non-diffusive Hamilton-Jacobi equation.
\end{abstract}

\vspace{2.0 cm}


\medskip


\newpage

\section{Introduction}

In this paper we study qualitative properties of the non-negative
continuous solutions to the following equation with singular
diffusion and gradient absorption
\begin{equation}\label{eq1}
\partial_t u -\Delta_{p}u+|\nabla u|^{q}=0, \quad (t,x)\in
Q_{\infty}:=(0,\infty)\times\real^N,
\end{equation}
where we consider $1<p<2$, $q>0$ and a non-negative initial
condition
\begin{equation}\label{a2}
u(0,x)=u_{0}(x), \quad x\in\RR^N.
\end{equation}
As usual, the $p$-Laplacian operator is defined by
\begin{equation*}
\Delta_{p}u=\hbox{div}(|\nabla u|^{p-2}\nabla u).
\end{equation*}
Equation \eqref{eq1}, when $p\in(1,2)$, is a quasilinear singular
diffusion equation (also known in the literature as the \emph{fast
$p$-Laplacian equation}), with a nonlinear absorption term depending
on the euclidean norm of the gradient. In recent years, both the
semilinear problem ($p=2$) and the degenerate diffusion-absorption
problem ($p>2$) have been investigated, with emphasis on the large
time behavior. It has been noticed that the asymptotic behavior as
$t\to\infty$ depends strongly on the value of $q>0$, and for $p=2$
there are many results available, see for example \cite{ATU04,
BKaL04, BLS01, BLSS02, BRV97, BGK04, GL07, Gi05}. From all these results,
an almost complete understanding of the large time behavior for the
semilinear case $p=2$ is now available. In particular, finite time
extinction takes place for $q\in (0,1)$ while the dynamics is either
solely dominated by the diffusion or is the result of a balance
between the diffusion and the absorption according to the value of
$q>1$.

More recently, the research has been extended to the degenerate case
$p>2$. In this range, the situation is very different: indeed, on
the one hand, the support of compactly supported solutions advances
in time with finite speed and interfaces appear  \cite{BtL08}. On
the other hand, there is a range of values of the parameter $q$,
namely $q\in (1,p-1]$, where the dynamics of \eqref{eq1}-\eqref{a2}
is solely governed by the gradient absorption \cite{ILV,LV07}, a
feature which cannot be observed in the semilinear case ($p=2$) for
$q>1$.

The purpose of this paper is to investigate the range $p\in (1,2)$,
called fast $p$-Laplacian diffusion, where the diffusion is no
longer degenerate but becomes singular when $\nabla u$ vanishes.
This case turns out to be more complicated and we first point out
that, even in the case of the diffusion equation
\begin{equation}\label{eq:diffusion}
\partial_t\Phi - \Delta_p \Phi=0 \quad \text{ in } \quad Q_\infty\,,
\end{equation}
important advances have been performed very recently, both in
constructing special solutions with optimal decay estimates, see
\cite{ISV, VazquezSmoothing} and in understanding regularity,
smoothing effects and other deep qualitative properties of the
solutions \cite{BIV}. All this previous knowledge is a good starting
point to investigate the competition between the fast $p$-Laplacian
diffusion and the gradient absorption terms. The behavior of
non-negative solutions $\Phi$ to the diffusion equation
\eqref{eq:diffusion} and of non-negative solutions $h$ to the
Hamilton-Jacobi equation
\begin{equation}\label{eq:hj}
\partial_t h + |\nabla h|^q=0 \quad \text{ in } \quad Q_\infty
\end{equation}
indeed differs markedly: in particular, starting from a compactly
supported initial condition, $\Phi$ becomes instantaneously
positive in $Q_\infty$ if $p\ge 2N/(N+1)$ while the support of $h$
stays the same for all times if $q>1$ or becomes empty after a
finite time if $q\in (0,1]$. It is thus of interest to figure out
how these two mechanisms compete in \eqref{eq1}.

\medskip

More specifically, the aim of this paper is to give a complete
picture of the qualitative properties of non-negative solutions to
\eqref{eq1}-\eqref{a2}, with respect to the following three types of
behaviors: either the solution remains positive in the limit, or it
decays to zero as $t\to\infty$ but is positive for finite times, or
finally it extinguishes after a finite time. In fact, we describe
the ranges, with respect to $p$ and $q$, where these phenomena
occur, and we also provide, in the cases where this is possible, a
quantitative measure of how the solution behaves, providing
estimates of decay rates or extinction rates.

The main tool for establishing such qualitative properties turns out
to be gradient estimates having generally the form
\begin{equation}\label{grad_est_abstract}
\|\nabla u^{\gamma}(t)\|_{\infty}\leq C\|u_0\|_{\infty}^{\delta} t^{-\beta},
\end{equation}
for suitable exponents $\gamma$, $\delta>0$, and $\beta>0$. Such
gradient estimates have been obtained in \cite{BL99,GGK03} for $p=2$
and $q>0$ and in \cite{BtL08} for $p>2$ and $q>1$ by a Bernstein
technique adapted from \cite{Be}, the exponent $\gamma$ depending on
$p$ and $q$ and ranging in $(0,1)$ for $p\ge 2$ and $q>1$. This last
property is of great interest as such estimates are clearly stronger
than an estimate on $\|\nabla u(t)\|_\infty$ and are at the basis of
the subsequent studies of the qualitative behavior of solutions to
\eqref{eq1} for $p\ge 2$. We shall establish similar gradient
estimates for \eqref{eq1} when $p$ and $q$ range in $(1,2)$ and
$(0,\infty)$, respectively. A particularly interesting new feature is
that the singular diffusion allows us to obtain gradient estimates
with negative exponents $\gamma$. As we shall see below, these
estimates have clearly a link with the positivity properties of the
solutions to \eqref{eq1} which are expected when the diffusion
dominates.

\medskip

\noindent \textbf{Notion of solution.} Owing to the nonlinear
reaction term $|\nabla u|^q$ involving the gradient of $u$, a
suitable notion of solution for Equation \eqref{eq1} is that of
\emph{viscosity solution}. Due to the singular character of
\eqref{eq1} at points where $\nabla u$ vanishes, the standard
definition of viscosity solution has to be adapted to deal with this
case \cite{IS, JLM, OS}. In fact, it requires to restrict the class
of comparison functions \cite{IS, OS} and we refer to
Definition~\ref{def:vs} for a precise definition. A remarkable
feature of this modified definition is that basic results about
viscosity solutions, such as comparison principle and stability
property, are still valid, see \cite[Theorem 3.9]{OS} (comparison
principle) and \cite[Theorem 6.1]{OS} (stability). The relationship
between viscosity solutions and other notions of solutions is
investigated in \cite{JLM}. From now on, by a solution to
\eqref{eq1}-\eqref{a2} we mean a viscosity solution in the sense of
Definition~\ref{def:vs} below.

\medskip

\noindent \textbf{Main results.}

For later use, we introduce the following notations for the critical
exponents
\begin{equation}\label{notations.crit}
p_c:=\frac{2N}{N+1}, \ p_{sc}:=\frac{2(N+1)}{N+3}, \
q_{\star}:=p-\frac{N}{N+1}
\end{equation}
and for several constants
\begin{equation}\label{notations.cons}
\begin{split}
k:=\frac{(2-p)[p(N+3)-2(N+1)]}{4(p-1)}, \ &\xi:=\frac{1}{q(N+1)-N},
\ \eta:=\frac{1}{N(p-2)+p},\\ &q_1:=\max\left\{p-1,\frac{N}{N+1}\right\},
\end{split}
\end{equation}
appearing frequently in our analysis. Throughout the paper, $C$,
$C'$, and $C_i$, $i\ge 1$, denote constants depending only on $N$,
$p$, and $q$. The dependence of these constants upon additional
parameters will be indicated explicitly.

Let us begin with basic decay estimates which are valid for general
non-negative Lipschitz continuous and integrable initial data
without any extra conditions.

\begin{theorem}\label{th:warmup}
Assume that
\begin{equation}\label{assump}
u_0\in L^1(\RR^N)\cap W^{1,\infty}(\RR^N), \quad u_0\geq0, \
u_0\not\equiv 0.
\end{equation}
Then there exists a unique non-negative (viscosity) solution $u$ to \eqref{eq1}-\eqref{a2} such that:
\begin{enumerate}[(i)]
\item if $p>p_c$ and $q>q_\star$, then
\begin{equation}
\|u(t)\|_\infty \le C\ \| u_0\|_1^{p\eta}\ t^{-N\eta}\,, \quad
t>0\,. \label{wu1}
\end{equation}
\item if $p>p_c$ and $q\in (N/(N+1),q_\star]$, then
\begin{equation}
\|u(t)\|_\infty \le C\ \| u_0\|_1^{q\xi}\ t^{-N\xi}\,, \quad t>0\,.
\label{wu2}
\end{equation}
\item  if $p> p_c$ and $q=N/(N+1)$ or $p=p_c$ and $q\ge p_c/2$, then
\begin{equation}
\|u(t)\|_\infty \le C'(u_0)\ e^{-C(u_0)t}\,, \quad t>0\,.
\label{wu3}
\end{equation}
\item  if $p\ge p_c$ and $q\in (0,N/(N+1))$ or $p\in (1,p_c)$, then there is $T_{\text{e}}>0$ depending only on $N$, $p$, $q$, and $u_0$ such that
\beqn u(t,x) \equiv 0\,, \quad (t,x)\in [T_{\text{e}},\infty)\times
\RR^N\,. \label{wu4} \eeqn
\end{enumerate}
\end{theorem}

Let us first mention that the main contribution of
Theorem~\ref{th:warmup} is not the existence and uniqueness of a
viscosity solution to \eqref{eq1}-\eqref{a2}, as the latter readily
follows from the comparison principle \cite[Theorem~3.9]{OS} while
the former is likely to be proved by Perron's method such as in
\cite[Section~4]{OS}. We shall however provide a proof in the final
section as it is needed in order to justify the derivation of the
gradient estimates stated below. Next, we notice that the decay
estimates \eqref{wu1} and \eqref{wu2} are also enjoyed by
non-negative and integrable solutions to \eqref{eq:diffusion} and
\eqref{eq:hj}, respectively. Since $t^{-N\xi}\le t^{-N\eta}$ for
$t\ge 1$ and $q<q_\star$,  Theorem~\ref{th:warmup} already uncovers
a dichotomy in the behavior of solutions to \eqref{eq1}-\eqref{a2}
for $p\ge p_c$ with a faster decay induced by the absorption term
for $q<q_\star$. This decay is even faster for $q\in (0,N/(N+1)]$.
Still, as we shall see now, more precise information can be obtained
for initial data with a fast decay at infinity and the first main
result of this paper is the following improvement of
Theorem~\ref{th:warmup} for $p>p_c$.

\begin{theorem}\label{th:idrext}
Assume that $u_0$ satisfies \eqref{assump}. Then the corresponding solution $u$ to \eqref{eq1}-\eqref{a2} satisfies:
\begin{enumerate}[(i)]
\item if $p\in (p_c,2)$, $q\in (p/2,q_\star)$, and there is $C_0>0$ such that
\beqn u_0(x) \le C_0\ |x|^{-(p-q)/(q-p+1)}\,, \quad x\in\RR^N\,,
\label{x1} \eeqn then \beqn t^{(N+1)(q_\star-q)/(2q-p)}\ \|u(t)\|_1
+ t^{(p-q)/(2q-p)}\ \|u(t)\|_\infty \le C(u_0)\,, \quad t>0\,.
\label{x2} \eeqn
\item if $p\in (p_c,2)$, $q=p/2$, and $u_0$ satisfies \eqref{x1}, then
\beqn \|u(t)\|_1 + \|u(t)\|_\infty \le C'(u_0)\ e^{-C(u_0)t}\,,
\quad t>0\,. \label{x3} \eeqn
\item if $p\in (p_c,2)$, $q\in (0,p/2)$, and there are $C_0>0$ and $Q>0$ such that
\beqn u_0(x) \le C_0\ |x|^{-(p-Q)/(Q-p+1)}\,, \quad x\in\RR^N\,,
\label{x4} \eeqn with $Q=q$ if $q\in (q_1,p/2)$ and $Q\in (
q_1,p/2)$ if $q\in (0,q_1]$. Then there is $T_{\text{e}}>0$
depending only on $N$, $p$, $q$, and $u_0$ such that \beqn u(t,x)
\equiv 0\,, \quad (t,x)\in [T_{\text{e}},\infty)\times \RR^N\,.
\label{x5} \eeqn
\end{enumerate}
\end{theorem}

Noting that $(p-q)/(2q-p)>N\xi$ for $q\in (0,q_\star)$, the decay
estimates obtained in Theorem~\ref{th:idrext} are clearly faster
than those of Theorem~\ref{th:warmup} for initial data decaying
sufficiently rapidly as $|x|\to\infty$.

Let us next notice that a very interesting point in the previous
theorem is the appearance of a new critical exponent for the
absorption, $q=p/2$, that in the slow-diffusion range $p>2$ did not
play any role. Moreover, this critical exponent is a branching point
for the behavior, as an interface between decay as $t\to\infty$ and
finite time extinction. It is worth mentioning that the
corresponding critical exponent for $p>2$ is $q=p-1$ and that we
have $p-1=p/2=1$ exactly when $p=2$.

Another interesting remark related to Theorem~\ref{th:idrext} is the
fact that, for $p\in[p_c,2)$ and $q\in[p/2,1)$, the diffusion
prevents extinction in finite time, see Proposition~\ref{pr:pos}
below. This is a feature which matches with the linear diffusion
case $p=2$, since, under suitable conditions on the initial data
$u_0$, finite time extinction could appear for any $q\in (0,1)$
\cite{BLS01,BLSS02, Gi05}.

\medskip

As mentioned above, the key technical tool for studying the large
time behavior of the solutions of \eqref{eq1} is the availability of
suitable gradient estimates, with abstract form
\eqref{grad_est_abstract}. Their proof relies on a Bernstein
technique borrowing ideas from \cite{Be} and, apart from their
technical interest in the proof of our main theorem, they are
interesting by themselves. Let us first denote the positivity set $\cp$ of
$u$ by
\begin{equation}\label{pos.set}
\cp:=\{(t,x)\in Q_{\infty}: \ u(t,x)>0\}.
\end{equation}

\begin{theorem}\label{th:grad1}
Let $p>p_c$ and $u_0$ satisfy \eqref{assump}. The corresponding
solution $u$ to \eqref{eq1}-\eqref{a2} satisfies the following
gradient estimates:

\begin{enumerate}[(i)]
\item for $q\in[1,\infty)$, we have
\begin{equation}\label{grad.est1}
\left|\nabla
u^{-(2-p)/p}(t,x)\right|\leq\left(\frac{2-p}{p}\right)^{(p-1)/p}\eta^{1/p}t^{-1/p},
\ (t,x)\in \cp.
\end{equation}

\item for $q\in[p/2,1)$, we have
\begin{equation}\label{grad.est2}
|\nabla u^{-(2-p)/p}(t,x)|\leq
C\left(\|u_0\|_{\infty}^{(2q-p)/p(p-q)}+t^{-1/p}\right),
\ (t,x)\in \cp.
\end{equation}

\item for $q\in(p-1,p/2)$, we have
\begin{equation}\label{grad.est3}
|\nabla u^{-(q-p+1)/(p-q)}(t,x)|\leq
C\left(1+\|u_0\|_{\infty}^{(p-2q)/p(p-q)}t^{-1/p}\right),
\ (t,x)\in \cp.
\end{equation}

\item for $q=p-1$, we have the logarithmic estimate
\begin{equation}\label{grad.est4}
|\nabla\log u(t,x)|\leq C\left(1+\|u_0\|_{\infty}^{
(2-p)/p}t^{-1/p}\right), \ (t,x)\in \cp.
\end{equation}

\item for $q\in(0,p-1)$, we have
\begin{equation}\label{grad.est5}
|\nabla u^{(p-q-1)/(p-q)}(t,x)|\leq
C\left(1+\|u_0\|_{\infty}^{(p-2q)/p(p-q)}t^{-1/p}\right), \ (t,x)\in
Q_\infty.
\end{equation}
\end{enumerate}
\end{theorem}

A striking feature in Theorem~\ref{th:grad1} is that in parts
(i)-(iii) gradients of negative powers of the solutions appear.
Besides being seemingly new, these estimates are rather unusual and
obviously stronger than an estimate for only $|\nabla u|$, which can
be easily deduced from them. They are valid only on the positivity
set of $u$ but, as we shall show below, $\cp$ coincides with
$Q_{\infty}$ when $p\geq p_c$ and $q\geq p/2$, and $\cp\subseteq
(0,T_e)\times\real^N$ for $1<p<p_c$ or $p_c\le p<2$ and $q<p/2$, for
some $T_e<\infty$.

\begin{remark}\label{rem:1} We actually prove a stronger result,
namely that, for any $\delta>0$,
$|\nabla(u+\delta)^{-(2-p)/p}(t,x)|$ (respectively
$|\nabla(u+\delta)^{-(2-p)/p}(t,x)|$,
$|\nabla(u+\delta)^{-(q-p+1)/(p-q)}(t,x)|$ and
$|\nabla\log(u+\delta)(t,x)|$) is bounded by the same right-hand
side as in \eqref{grad.est1} (respectively \eqref{grad.est2},
\eqref{grad.est3} and \eqref{grad.est4}) for all $(t,x)\in
Q_{\infty}$. For instance, for $q\in[1,\infty)$ we have
\begin{equation}\label{grad.est.strong}
\left|\nabla(u+\delta)^{-(2-p)/p}(t,x)\right|\leq
\left(\frac{2-p}{p}\right)^{(p-1)/p}\eta^{1/p}t^{-1/p},
\ (t,x)\in Q_{\infty}.
\end{equation}
As the right-hand side of \eqref{grad.est.strong} does not depend on
$\delta>0$, we deduce \eqref{grad.est1} by letting $\delta\to 0$
wherever it is possible, that is in $\mathcal{P}$.
\end{remark}

These gradient estimates will be used in the sequel to prove parts
of Theorem~\ref{th:idrext}. Their
proof is divided into two parts and performed in Sections~\ref{subsec.grad.est1} and~\ref{subsec.grad.est2}.

We obtain similar gradient estimates for $p=p_c$ and $p<p_c$. In the
case $p=p_c$ being a critical exponent, some logarithmic
corrections appear in the gradient estimates; they are gathered in
the following result, that is proved in Section~\ref{subsec.grad.est3}. Notice that, as $p_c=1$ in one space
dimension, the next theorem is only valid for $N\geq2$.

\begin{theorem}\label{th:grad2}
Let $p=p_c$ and $u_0$ satisfy \eqref{assump}. The corresponding
solution $u$ to \eqref{eq1}-\eqref{a2} satisfies the following
gradient estimates:

\begin{enumerate}[(i)]

\item for $q\geq1$ and $(t,x)\in \cp$, we have
\begin{equation}\label{grad.est6}
|\nabla u^{-1/N}(t,x)|\leq C \left( \log\left(\frac{e
\|u_0\|_\infty}{u(t,x)} \right) \right)^{1/p_c}t^{-1/p_c}.
\end{equation}

\item for $q\in(N/(N+1),1)$ and $(t,x)\in\cp$, we have
\begin{equation}\label{grad.est7}
|\nabla u^{-1/N}(t,x)|\leq C\left(\|u_0\|_{\infty}^{
1/N\xi(p_c-q)}+t^{-1/p_c}\right)\ \left( \log\left(\frac{e
\|u_0\|_\infty}{u(t,x)} \right) \right)^{1/p_c}.
\end{equation}

\item for $q=N/(N+1)=p_c/2$ and $(t,x)\in\cp$, we have
\begin{equation}\label{grad.est.ex}
|\nabla u^{-1/N}(t,x)|\leq C\ \left( \log\left(\frac{e
\|u_0\|_\infty}{u(t,x)} \right) \right)^{2/p_c} \left( 1+t^{-1/p_c}
\right).
\end{equation}

\item for $q\in(0,N/(N+1))$, the previous gradient estimates
\eqref{grad.est3}, \eqref{grad.est4} and \eqref{grad.est5} still
hold true.
\end{enumerate}
\end{theorem}
\begin{remark}\label{rem:3} Similarly to the case $p>p_c$ (recall Remark~\ref{rem:1}), given $\delta>0$, the estimates \eqref{grad.est6}-\eqref{grad.est.ex} are true for all $(t,x)\in Q_\infty$ provided that $u(t,x)$ is replaced by $u(t,x)+\delta$ on both sides of the inequalities.
\end{remark}

In the range $p<p_c$, the situation becomes more technical and more
involved, and apparently there is a new critical exponent coming
from the diffusion that plays a role, $p_{sc}=2(N+1)/(N+3)$. We can
still establish gradient estimates for this range, but it requires
to handle separately several cases according to the value of $q$.
Since they are not used afterwards, we do not state nor prove them
but refer the interested reader to Section~\ref{subsec.grad.est4}
where we provide a proof only for a limited range of $q$, namely,
$q\geq 1-k$.

Finally, another useful gradient estimate is the one which retains
only the influence of the Hamilton-Jacobi term:

\begin{theorem}\label{th:grad3}
Let $p\in[p_c,2)$ and $u_0$ satisfy \eqref{assump}. The
corresponding solution $u$ to \eqref{eq1}-\eqref{a2} satisfies  the
following gradient estimates: if $q\in (0,1)$, we have
\begin{equation}\label{grad.estHJ}
|\nabla u(t,x)|\leq C\|u_0\|_{\infty}^{1/q}t^{-1/q}, \ (t,x)\in
Q_{\infty},
\end{equation}
while, if $q>1$, we have a slightly better formulation:
\begin{equation}\label{grad.estHJ2}
\left|\nabla u^{(q-1)/q}(t,x)\right|\leq\frac{1}{q}(q-1)^{(q-1)/q}t^{-1/q},
\ (t,x)\in Q_{\infty}.
\end{equation}
\end{theorem}

These estimates are proved by similar modified Bernstein
techniques, but their main difference with respect to the previous
ones is that it is the term coming from the
diffusion which is simply discarded. They actually hold in more general ranges
of $p$ as we can deduce by analyzing their proof in Section~\ref{subsec.grad.HJ}.

\medskip

Having discussed the occurrence of finite time extinction in
Theorems~\ref{th:warmup} and~\ref{th:idrext} and obtained gradient
estimates valid on the positivity set \eqref{pos.set} of $u$, we
finally turn to the positivity issue: we first observe that the
$L^1$-norm of solutions $u$ to \eqref{eq1}-\eqref{a2} is
non-increasing. It thus has a limit as $t\to\infty$ which is
non-negative and it is natural to wonder whether the absorption term
may drive it to zero as $t\to\infty$ or not. This question is
obviously only meaningful for $p\ge p_c$ for which there is no
extinction for the diffusion equation \eqref{eq:diffusion} but
conservation of mass \cite{HV81}. In this direction, we also prove
the following positivity result that completes the panorama given in
Theorem~\ref{th:idrext}.

\begin{proposition}\label{pr:pos}
Let $p\in [p_c,2)$, $u_0$ satisfy \eqref{assump}, and $u$ be the
solution to \eqref{eq1}-\eqref{a2}.
\begin{enumerate}[(1)]
\item If either $p>p_c$ and $q\geq p/2$ or $p=p_c$ and $q>p_c/2$, then $\|u(t)\|_{1}>0$ for all $t\geq0$
and the positivity set satisfies $\cp=Q_{\infty}$.

\item We have $\lim\limits_{t\to\infty}\|u(t)\|_1>0$ if and
only if $q>q_{\star}$.
\end{enumerate}
\end{proposition}

Thanks to Theorem~\ref{th:idrext} and Proposition~\ref{pr:pos}, we
thus have a clear separation between positivity and finite time
extinction, the latter occurring when either $p\ge p_c$ and $q\in
(0,p/2)$ or $p\in (1,p_c)$ while the former is true in $Q_\infty$
for $p\ge p_c$ and $q\ge p/2$. Let us emphasize that, for $p\in
[p_c,2)$ and $q\in [p/2,1)$, the diffusion term prevents the finite
time extinction that would occur in the absence of diffusion.
Table~\ref{tab:1} provides a summary of the outcome of this paper.

\begin{table}[h]\label{tab:1}
\begin{center}
\begin{tabular}{|c|c|c|c|c|}
\hline
 & $0<q<p/2$ & $q=p/2$ & $p/2<q<q_\star$ & $q_\star\le q$ \\
 \hline
 $p\in [p_c,2)$ & $\text{extinction}$ & $\begin{array}{c}\text{positivity} \\ \hline \text{exponential} \\ \text{decay} \end{array}$ & $\begin{array}{c}\text{positivity} \\ \hline \text{fast algebraic} \\ \text{decay} \end{array}$ & $\begin{array}{c}\text{positivity} \\ \hline \text{diffusion} \\ \text{decay} \end{array}$\\
 \hline
 $p\in (1,p_c)$ & $\begin{array}{c} \\ \text{extinction}\\ \hfill \end{array}$ & $\begin{array}{c} \\ \text{extinction}\\ \hfill \end{array}$ & $\begin{array}{c} \\ \text{extinction}\\ \hfill \end{array}$ & $\begin{array}{c} \\ \text{extinction}\\ \hfill \end{array}$ \\
 \hline
\end{tabular}
\end{center}
\caption{\label{table1} Behavior of $u$ for initial data decaying sufficiently fast at infinity }
\end{table}

\noindent \textbf{Organization of the paper}. A formal proof of the
gradient estimates for solutions to \eqref{eq1} is given in
Section~\ref{sec.grad}, which is divided into several subsections
according to the range of the exponents $p$ and $q$. Then, a
rigorous approach by approximation and regularization, completing
the formal one and settling also the existence and uniqueness of
solutions to \eqref{eq1}-\eqref{a2} is appended, due to its highly
technical character, see Section~\ref{se:wp}. We prove
Theorem~\ref{th:warmup} in Section~\ref{sec.decay}. Before proving
our main Theorem~\ref{th:idrext}, we devote Section~\ref{sec.large}
to the behavior of the $L^1$-norm of $u$ as $t\to\infty$ and to the
positivity issue as well. Finally, we prove our main
Theorem~\ref{th:idrext}, together with Proposition~\ref{pr:pos}, in
Section~\ref{se:idrae}.

\section{Gradient estimates}\label{sec.grad}

As already mentioned, the proof of the gradient estimates
relies on a Bernstein technique \cite{Be}, also used in
\cite{BtL08,BL99,GGK03} for $p\ge 2$, but in the case $p\in (1,2)$ the
technical details are quite different. We first have the following
technical general lemma.

\begin{lemma}\label{techlemma}
Let $p\in(1,2)$, $q>0$, and consider a $C^3$-smooth monotone function
$\varphi$. Set $v:=\varphi^{-1}(u)$ and $w:=|\nabla v|^2$, where $u$
is a solution of \eqref{eq1}. Then, the function $w$ satisfies the
following differential inequality:
\begin{equation}\label{sub_diff}
\partial_{t}w-Aw-B\cdot\nabla w+R\leq 0 \quad \mbox{ in }\;\; Q_\infty,
\end{equation}
where $B$ is given in \cite[Appendix~A, Eq.~(A.2)]{BtL08}, and
\begin{equation}\label{A}
Aw:=|\nabla u|^{p-2}\Delta w+(p-2)|\nabla u|^{p-4}(\nabla
u)^{t}D^{2}w\nabla u,
\end{equation}
\begin{equation}\label{R}
R:=2(p-1)R_{1}w^{(2+p)/2}+2(q-1)R_{2}w^{(2+q)/2},
\end{equation}
where $R_1$ and $R_2$ are given by
\begin{equation}\label{R1}
R_1:=|\varphi'|^{p-2}\left(k\left(\frac{\varphi''}{\varphi'}\right)^2-\left(\frac{\varphi''}{\varphi'}\right)'\right)
\end{equation}
(recall that $k$ is defined in \eqref{notations.cons}) and
\begin{equation}\label{R2}
R_2:=|\varphi'|^{q-2}\varphi''.
\end{equation}
\end{lemma}

We do not recall the precise form of $B$, since it is complicated
and not needed in the sequel.
\begin{proof}
We begin with Lemma~2.1 in \cite{BtL08}, which, by examining
carefully the proof, holds true for monotone functions $\varphi$
(not only for increasing functions, as stated in \cite{BtL08}). We
obtain the differential inequality
\begin{equation*}
\partial_{t}w-Aw-B\cdot\nabla
w+2\tilde{R}_{1}w^{2}+2\tilde{R}_{2}w\leq0,
\end{equation*}
where $A$ and $B$ have the form given in \eqref{A} and in \cite[Eq.
(A.2)]{BtL08}, respectively, and
\begin{equation*}
\tilde{R}_1:=-a\left(\frac{\varphi''}{\varphi'}\right)'-\left((N-1)\frac{(a')^2}{a}+4a''\right)(\varphi' \varphi'')^{2} w^{2}
-2a' w(2(\varphi'')^{2}+\varphi' \varphi'''),
\end{equation*}
\begin{equation*}
\tilde{R}_2:=\frac{\varphi''}{(\varphi')^{2}} \left( 2b' (\varphi')^{2}w-b \right),
\end{equation*}
the dependence of $a$, $a'$, $a''$, $b$, $b'$ on $\varphi'(v)^2 w$ and of
$\varphi$ and its derivatives on $v$ being omitted. In our case
$a(r)=r^{(p-2)/2}$, $b(r)=r^{q/2}$. Using these formulas for $a$ and
$b$ and the identity
\begin{equation*}
\varphi' \varphi'''=\left(\frac{\varphi''}{\varphi'}\right)' (\varphi')^{2}+(\varphi'')^{2},
\end{equation*}
we compute $\tilde{R_1}$ and $\tilde{R_2}$ and obtain
\begin{equation*}
\begin{split}
-\tilde{R}_1&=(p-1)|\varphi'|^{p-2}\left(\frac{\varphi''}{\varphi'}\right)' w^{(p-2)/2} \\
& +(p-2)\left[p-1+\frac{(N-1)(p-2)}{4}\right] (\varphi'')^2 \left|\varphi'\right|^{p-4} w^{(p-2)/2}\\
&=(p-1)|\varphi'|^{p-2}w^{(p-2)/2}\left[\left(\frac{\varphi''}{\varphi'}\right)'-k\left(\frac{\varphi''}{\varphi'}\right)^{2}\right]
=-(p-1)w^{(p-2)/2}R_1.
\end{split}
\end{equation*}
and
\begin{equation*}
\tilde{R}_2=\frac{\varphi''}{{(\varphi')}^{2}}\left(q\left|\varphi'\right|^{q}w^{(q-2)/2}w-\left|\varphi'\right|^{q}w^{q/2}\right)
=(q-1)R_{2}w^{q/2}.
\end{equation*}
arriving to the formula \eqref{R2}. Let us notice that this is still
a formal proof, since \cite[Lemma~2.1]{BtL08} requires $a$ and $b$
to be $C^2$-smooth, and our choices are not. For a rigorous proof,
we have to approximate $a$ and $b$ by their regularizations
\begin{equation*}
a_{\e}(r):=(r+\e^2)^{(p-2)/2}, \quad
b_{\e}(r):=(\e^2+r)^{q/2}-\e^q, \ \e>0,
\end{equation*}
and pass to the limit as $\e\to 0$, see Section~\ref{se:wp}.
\end{proof}
We also introduce the function $\varrho:=1/\psi'$, where
$\psi:=\varphi^{-1}$. We have
$$\varphi'(v)=\varrho(u), \quad
\varphi''(v)=(\varrho\varrho')(u),$$ hence, by straightforward
calculations, we obtain the following alternative formulas for $R_1$
and $R_2$:
\begin{equation}\label{altR1}
R_1=|\varrho(u)|^{p-2}\left(k(\varrho'(u))^2-(\varrho\varrho'')(u)\right)
\end{equation}
and
\begin{equation}\label{altR2}
R_2=|\varrho(u)|^{q-2}\varrho(u)\varrho'(u).
\end{equation}
We now choose in an appropriate way $\varrho$ in equations
\eqref{altR1} and \eqref{altR2}, in order to have either $R_1=1$,
$R_2=1$ or $R_1=R_2$. In this way we obtain gradient estimates in
the form of estimates for the function $w$ in the notations of Lemma~\ref{techlemma}.

Let us notice at that point that, if we take $\varrho(z)\equiv1$, we
have $R_1=R_2=0$ and $\varphi=\psi=\hbox{Id}$; thus, $w=|\nabla
u|^2$ satisfies the differential inequality
\begin{equation*}
Lw:=\partial_{t}w-Aw-B\cdot\nabla w\leq0 \ \hbox{in} \ Q_{\infty}.
\end{equation*}
Since $w(0)\leq\|\nabla u_0\|_{\infty}^2$ and the constant function
$\|\nabla u_0\|_{\infty}^2$ is a solution for the operator $L$, by
comparison we obtain \beqn \|\nabla u(t)\|_\infty \le \|\nabla
u_0\|_\infty\,, \qquad t\ge 0\,.\label{prunelle} \eeqn

\subsection{Gradient estimates for $p>p_c$ and
$q\ge p/2$}\label{subsec.grad.est1}

For this range of parameters, we choose
\begin{equation}
\varrho(z)=\left(\frac{p^2}{2(2k+p-2)}\right)^{1/p}z^{2/p},
\end{equation}
after noticing that
\begin{equation}
2k+p-2=\frac{(2-p)(N+1)(p-p_c)}{2(p-1)}>0.
\end{equation}
Then it is immediate to check that $R_1=1$ (in fact this is the way
we discover this choice of $\varrho$) and
\begin{equation*}
R_2=\frac{2}{p}\left(\frac{p^2}{2(2k+p-2)}\right)^{q/p}u^{(2q-p)/p}\geq 0,
\end{equation*}
hence
\begin{equation*}
R=2(p-1)w^{(p+2)/2}+\frac{4(q-1)}{p}\left(\frac{p^2}{2(2k+p-2)}\right)^{q/p} u^{(2q-p)/p} w^{(q+2)/2}.
\end{equation*}

\medskip

\noindent \textbf{Case 1.} For $q\geq 1$, $(q-1)R_2\geq0$, so that
$R\geq2(p-1)w^{(p+2)/2}$ and therefore
\begin{equation}\label{diff.ineq}
Lw:=\partial_{t}w-Aw-B\cdot\nabla w+2(p-1)w^{(p+2)/2}\leq 0.
\end{equation}
Once established the differential inequality \eqref{diff.ineq}, the
next step (that will be also used in the other cases) is to find a
supersolution to the differential inequality \eqref{diff.ineq}
depending only on time, in this way avoiding the terms with the
complicated forms of $A$ and $B$. In our case, we notice that
$W(t):=(p(p-1)t)^{-2/p}$ is a supersolution and conclude that
\begin{equation*}
|\nabla v(t,x)|\leq (p(p-1)t)^{-1/p}, \ (t,x)\in Q_{\infty}.
\end{equation*}
But $v=\psi(u)$, hence $\nabla v=\psi'(u)\nabla u=\nabla
u/\varrho(u)$; thus, substituting the value of $\varrho$, we obtain
the inequality
\begin{equation*}
|\nabla
u(t,x)|\,u(t,x)^{-2/p}\leq\left[\frac{p^2}{2(2k+p-2)p(p-1)t}\right]^{1/p},
\end{equation*}
or equivalently \eqref{grad.est1}.

\medskip

\noindent \textbf{Case 2.} For $q\in[p/2,1)$, the term coming from
$R_2$ becomes negative and cannot be omitted. Instead, we will get
the gradient estimate by compensating its negative effect with the
positive term coming from $R_1$. Since $u(t,x)\leq\|u_0\|_{\infty}$
for any $(t,x)\in Q_{\infty}$ and $2q-p>0$, we have
\begin{equation*}
\begin{split}
R&=2(p-1)w^{(p+2)/2}-\frac{4(1-q)}{p}\left(\frac{p^2}{2(2k+p-2)}\right)^{q/p} u^{(2q-p)/p} w^{(q+2)/2}\\
&\geq
2(p-1)w^{(q+2)/2}\left[w^{(p-q)/2}-\frac{4(1-q)}{p}\left(\frac{p^2}{2(2k+p-2)}\right)^{q/p}\|u_0\|_{\infty}^{(2q-p)/p} \right],
\end{split}
\end{equation*}
hence
\begin{equation}
Lw:=\partial_{t}w-Aw-B\cdot\nabla
w+2(p-1) w^{(q+2)/2} \left( w^{(p-q)/2}-c_1 \right) \leq 0,
\end{equation}
where
\begin{equation*}
c_1:=\frac{4(1-q)}{p}\left(\frac{p^2}{2(2k+p-2)}\right)^{q/p} \|u_0\|_{\infty}^{(2q-p)/p}>0.
\end{equation*}
In a similar way as in the case $q\geq1$, we notice that the
function $W(t):=(2c_1)^{2/(p-q)}+(p(p-1)t/2)^{-2/p}$ is a
supersolution for the partial differential operator $L$, hence
\begin{equation*}
|\nabla
v(t,x)|\leq(2c_1)^{1/(p-q)}+\left(\frac{2}{p(p-1)t}\right)^{1/p},
\ (t,x)\in Q_{\infty}.
\end{equation*}
Since $|\nabla v|=|\nabla u|/\varrho(u)$, we deduce that there
exists a constant $C>0$ such that
\begin{equation*}
\left| \nabla u^{-(2-p)/p}(t,x) \right|\leq
C\left(\|u_0\|_{\infty}^{(2q-p)/p(p-q)}+t^{-1/p}\right),
\ (t,x)\in Q_{\infty},
\end{equation*}
as stated in \eqref{grad.est2}.

\subsection{Gradient estimates for $p>p_c$ and
$q<p/2$}\label{subsec.grad.est2}

In this case, we choose
\begin{equation}
\varrho(z)=\left(\frac{p-q}{k+p-q-1}\right)^{1/(p-q)}z^{1/(p-q)},
\end{equation}
noticing that
\begin{equation*}
k+p-q-1=\frac{p}{2}-q+\frac{2k+p-2}{2}=\frac{p}{2}-q+\frac{(2-p)(N+1)(p-p_c)}{4(p-1)}>0.
\end{equation*}
By straightforward calculations, it is immediate to check that
\begin{equation*}
R_1=R_2=\varrho(u)^{q-1}\varrho'(u)=\frac{1}{p-q}\left(\frac{p-q}{k+p-q-1}\right)^{q/(p-q)} u^{(2q-p)/(p-q)}\ge 0,
\end{equation*}
so that
\begin{equation*}
R=2(p-1)R_2w^{(q+2)/2}\left(w^{(p-q)/2}-\frac{1-q}{p-1}\right).
\end{equation*}
It follows that
\begin{equation}
Lw:=\partial_{t}w-Aw-B\cdot\nabla
w+2(p-1)R_2w^{(q+2)/2}\left(w^{(p-q)/2}-\frac{1-q}{p-1}\right)\leq 0.
\end{equation}
We next look for a supersolution of the form
$W(t)=(2(1-q)/(p-1))^{2/(p-q)}+Kt^{-2/p}$, with $K$ to be chosen
depending on $p$, $q$, $N$, and $\|u_0\|_{\infty}$. Taking into
account that $u(t,x)\leq\|u_0\|_{\infty}$ for any $(t,x)\in
Q_{\infty}$ and $2q-p<0$, we have
\begin{equation*}
R_2\geq\frac{1}{p-q}\left(\frac{p-q}{k+p-q-1}\right)^{q/(p-q)}\|u_0\|_{\infty}^{(2q-p)/(p-q)}
\end{equation*}
and
\begin{equation*}
\begin{split}
LW&=-\frac{2}{p}Kt^{-1-(2/p)}+(p-1)R_2\left[W^{(p+2)/2}+W^{(q+2)/2}
\left(W^{(p-q)/2}-2\frac{1-q}{p-1}\right)\right]\\
&\geq-\frac{2}{p}Kt^{-1-(2/p)}+(p-1)R_2K^{(p+2)/2}t^{-1-(2/p)}\\
&\geq\frac{2K}{p}\left[\frac{p(p-1)}{2(p-q)}\left(\frac{p-q}{k+p-q-1}\right)^{q/(p-q)} \|u_0\|_{\infty}^{(2q-p)/(p-q)}K^{p/2}-1\right]t^{-1-(2/p)}
\end{split}
\end{equation*}
hence, we find that $LW\geq 0$ provided that
$K=C\|u_0\|_{\infty}^{2(p-2q)/p(p-q)}$ for some sufficiently large
constant $C$. With this choice of $K$, the function $W$ becomes a
supersolution for $L$, and the comparison principle gives
\begin{equation*}
|\nabla v(t,x)|\leq
C\left(1+\|u_0\|_{\infty}^{(p-2q)/p(p-q)}t^{-1/p}\right),
\ (t,x)\in Q_\infty,
\end{equation*}
or equivalently
\begin{equation}\label{part.grad.est2}
|\nabla u(t,x)|u(t,x)^{-1/(p-q)}\leq
C\left(1+\|u_0\|_{\infty}^{(p-2q)/p(p-q)}t^{-1/p}\right), \ (t,x)\in
Q_\infty.
\end{equation}
Thus, we have a discussion with respect to the sign of $p-1-q$.
Indeed, if $q\in(p-1,p/2)$, we have
\begin{equation*}
|\nabla u^{-(q-p+1)/(p-q)}(t,x)|\leq
C\left(1+\|u_0\|_{\infty}^{(p-2q)/p(p-q)}t^{-1/p}\right),
\ (t,x)\in Q_\infty.
\end{equation*}
If $q=p-1$, we have the logarithmic estimate
\begin{equation*}
|\nabla\log u(t,x)|\leq
C\left(1+\|u_0\|_{\infty}^{(p-2q)/p(p-q)}t^{-1/p}\right),
\ (t,x)\in Q_\infty,
\end{equation*}
and if $q\in(0,p-1)$ we obtain a positive power estimate
\begin{equation*}
|\nabla u^{(p-q-1)/(p-q)}(t,x)|\leq
C\left(1+\|u_0\|_{\infty}^{(p-2q)/p(p-q)}t^{-1/p}\right),
\ (t,x)\in Q_\infty.
\end{equation*}
This completes the proof of Theorem~\ref{th:grad1}.

\subsection{Gradient estimates for $p=p_c$ (and $N\geq 2$)}\label{subsec.grad.est3}

\noindent \textbf{Case 1.} Let us consider first $q>p_c/2=N/(N+1)$.
In this case, the constant $k$ defined in \eqref{notations.cons} is
given by $k=(2-p)/2=1/(N+1)$. By analogy with some gradient
estimates obtained by Hamilton in \cite{Ha93} for the heat equation,
we choose the following function:
\begin{equation*}
\varrho(u)=u^{(N+1)/N}(\log M-\log u)^{(N+1)/2N}, \
M=e\|u_0\|_{\infty}.
\end{equation*}
Let us notice first that $\log M-\log u\geq1$. Then, we obtain
\begin{equation*}
\varrho'(u)=\frac{N+1}{2N} u^{1/N}\left[2\left( \log\frac{M}{u} \right)^{(N+1)/2N} - \left( \log\frac{M}{u} \right)^{-(N-1)/2N}\right]
\end{equation*}
and
\begin{equation*}
\begin{split}
\varrho''(u) =u^{-(N-1)/N} & \left[\frac{N+1}{N^2} \left( \log\frac{M}{u} \right)^{(N+1)/2N} -\frac{(N+1)(N+2)}{2N^2} \left( \log\frac{M}{u} \right)^{-(N-1)/2N} \right. \\
& \quad \left.-\frac{(N+1)(N-1)}{4N^2} \left( \log\frac{M}{u} \right)^{-((N-1)/2N)-1}\right] .
\end{split}
\end{equation*}
Hence, after an easy calculation, we have
\begin{equation*}
k(\varrho'(u))^2-\varrho(u)\varrho''(u) =u^{2/N}\left[\frac{N+1}{2N} \left( \log\frac{M}{u} \right)^{1/N}+\frac{N+1}{4N} \left( \log\frac{M}{u} \right)^{-(N-1)/N} \right],
\end{equation*}
which implies that
\begin{equation*}
R_1=\frac{N+1}{2N}+\frac{N+1}{4N}(\log M-\log
u)^{-1}\geq\frac{N+1}{2N}.
\end{equation*}
On the other hand, calculating $R_2$, we find:
\begin{equation*}
R_2=\frac{N+1}{2N} u^{(q(N+1)-N)/N}\  \left[ 2 \left( \log\frac{M}{u} \right)^{(N+1)q/2N} - \left( \log\frac{M}{u} \right)^{((N+1)q-2N)/2N} \right]>0,
\end{equation*}
since $(\log M-\log u)^{-1}\leq1<2$. Following the same division
into cases with respect to $q$, we assume first that $q\geq 1$. In
this case, we can simply omit the term coming from $R_2$, since
$(q-1)R_2\ge0$, and end up with
\begin{equation*}
R\geq\frac{N-1}{N} w^{(p_c+2)/2}.
\end{equation*}
Therefore
\begin{equation}
Lw:=\partial_{t}w-Aw-B\cdot\nabla
w+\frac{N-1}{N} w^{(p_c+2)/2}\leq0.
\end{equation}
Noticing that $W(t)=[(N+1)/(N-1)t]^{2/p_c}$ is a supersolution for
$L$, we obtain that
\begin{equation*}
|\nabla v(t,x)|\leq\left(\frac{N+1}{(N-1)t}\right)^{1/p_c}.
\end{equation*}
Coming back to the function $u$, this means
\begin{equation}
|\nabla
u^{-1/N}(t,x)|\leq\frac{1}{N}\left(\frac{N+1}{N-1}\right)^{(N+1)/2N}(\log
M-\log u(t,x))^{(N+1)/2N}t^{-(N+1)/2N}.
\end{equation}

\medskip

\noindent \textbf{Case 2.} Consider next $q\in(p_c/2,1)$, In this
case, we have to use again the strategy of compensation as in
Section~\ref{subsec.grad.est1}. First of all, we need to estimate
$R_2$ from above. To this end, since
$1/N\xi=[q(N+1)-N]/N<q(N+1)/2N$, we note that the function
$$
z\mapsto z^{(q(N+1)-N)/N}(\log M-\log z)^{q(N+1)/2N}
$$
attains its maximum over $(0,\|u_0\|_{\infty})$ at
$\|u_0\|_{\infty}e^{-(N\xi-1)/2}<\|u_0\|_{\infty}$. We deduce that
\begin{equation*}
R_2\leq\frac{N+1}{N}u^{(q(N+1)-N)/N} (\log M-\log u)^{(N+1)q/2N}\leq
C_1\|u_0\|_{\infty}^{(q(N+1)-N)/N},
\end{equation*}
hence
\begin{equation*}
\begin{split}
R & \geq\frac{N-1}{N} w^{(p_c+2)/2}-2(1-q) C_1\|u_0\|_{\infty}^{(q(N+1)-N)/N} w^{(q+2)/2} \\
& =\frac{N-1}{N}w^{(q+2)/2}\left(w^{(p_c-q)/2}-c_2\right),
\end{split}
\end{equation*}
where
$$c_2=\frac{2N(1-q)C_1}{N-1}\|u_0\|_{\infty}^{(q(N+1)-N)/N}.$$ We
now proceed as in Section~\ref{subsec.grad.est1} and notice that
$W(t)=(2c_2)^{2/(p_c-q)}+[2(N+1)/(N-1)t]^{2/p_c}$ is a
supersolution. By the comparison principle we obtain
\begin{equation*}
|\nabla
v(t,x)|\leq(2c_2)^{1/(p_c-q)}+\left(\frac{2(N+1)}{(N-1)t}\right)^{1/p_c}.
\end{equation*}
Going back to the definition of $u$, we find that
\begin{equation*}
\frac{|\nabla u(t,x)|}{\varrho(u(t,x))}\leq
C\left(\|u_0\|_{\infty}^{(q(N+1)-N)/N(p_c-q)}+t^{-(N+1)/2N} \right),
\end{equation*}
from which we deduce easily \eqref{grad.est7}, taking into account
the definition of $\varrho$.

Let us remark that this is an extension of the estimates that we
obtain for $p>p_c$ and $q>p/2$, since for $p=p_c$, we have
$(2-p)/p=1/N$. Thus the negative power of the gradient is the same
and the powers of $t$ and $\|u_0\|_{\infty}$ in the right-hand side
are also the same. The presence of the logarithmic corrections is
the mark of the critical exponent.

\medskip

\noindent \textbf{Case 3.} We now consider the case
$q=p_c/2=N/(N+1)$ and choose
\begin{equation*}
\varrho(u)=u^{(N+1)/N}(\log M-\log u)^{(N+1)/N}, \
M=e\|u_0\|_{\infty}.
\end{equation*}
Then
\begin{equation*}
\varrho'(u)=\frac{N+1}{N}u^{1/N}\left[(\log M-\log
u)^{(N+1)/N}-(\log M-\log u)^{1/N}\right]
\end{equation*}
and
\begin{equation*}
\varrho''(u)=\frac{N+1}{N^2}u^{-(N-1)/N}\left[\left(\log\frac{M}{u}
\right)^{(N+1)/N}-(N+2)\left(\log\frac{M}{u}\right)^{1/N}+\left(\log\frac{M}{u}\right)^{-(N-1)/N}\right].
\end{equation*}
Thus, after straightforward calculations, we obtain
\begin{equation*}
R_1=\frac{N+1}{N}(\log M-\log u), \quad R_2=\frac{N+1}{N}(\log
M-\log u-1).
\end{equation*}
Therefore
\begin{equation*}
\begin{split}
R&=\frac{2(N-1)}{N}(\log M-\log u) w^{(p_c+2)/2}-\frac{2}{N}(\log
M-\log u)w^{(q+2)/2}+\frac{2}{N}w^{(q+2)/2}\\
&\geq\frac{1}{N}(\log M-\log
u)\left[2(N-1)w^{(p_c+2)/2}-2w^{(q+2)/2}\right],
\end{split}
\end{equation*}
and
\begin{equation*}
Lw:=\partial_{t}w-Aw-B\cdot\nabla w+\frac{1}{N}(\log M-\log
u)\left[2(N-1)w^{(p_c+2)/2}-2w^{(q+2)/2}\right]\leq 0.
\end{equation*}
As a supersolution, we take
$$W(t)=\left(\frac{2}{N-1}\right)^{2(N+1)/N}+\left(\frac{N+1}{(N-1)t}\right)^{(N+1)/N}$$
and deduce, recalling that $N\geq2$ and that $\log M-\log u\geq1$:
\begin{equation*}
\begin{split}
LW(t)&=-\left(\frac{N+1}{N-1}\right)^{(N+1)/N}\frac{N+1}{N}t^{-(2N+1)/N}+\frac{N-1}{N}(\log
M-\log u)W(t)^{(2N+1)/(N+1)}\\&+\frac{1}{N}(\log M-\log
u)W(t)^{(3N+2)/2(N+1)}\left((N-1)W(t)^{N/2(N+1)}-2\right)\\
&\geq-\left(\frac{N+1}{N-1}\right)^{(N+1)/N}\frac{N+1}{N}t^{-(2N+1)/N}+\frac{N-1}{N}\left(\frac{N+1}{(N-1)t}\right)^{(2N+1)/N}=0.
\end{split}
\end{equation*}
The comparison principle gives
$$
|\nabla
v(t,x)|\leq\left(\frac{2}{N-1}\right)^{(N+1)/N}+\left(\frac{N+1}{(N-1)t}\right)^{(N+1)/2N},
$$
which implies \eqref{grad.est.ex}.

\medskip

\noindent \textbf{Case 4.} Finally, for $p=p_c$ and $q\in(0,p/2)$,
we notice that $k+p_c-q-1=(p_c-2q)/2>0$, hence we proceed as in
Section~\ref{subsec.grad.est2}. The estimates \eqref{grad.est3},
\eqref{grad.est4} and \eqref{grad.est5} hold according to whether
$q\in(p_c-1,p_c/2)$, $q=p_c-1$ or $q\in(0,p_c-1)$. This ends the
proof of Theorem~\ref{th:grad2}.

\subsection{Gradient estimates for $p<p_c$ and
$q\geq1-k$}\label{subsec.grad.est4}

We want now to follow the same idea as in
Section~\ref{subsec.grad.est1} and look for a function $\varrho$
such that $R_1=1$, that is, $\varrho$ is a solution of the following
ordinary differential equation:
\begin{equation}\label{ODE}
k(\varrho')^2-\varrho\varrho''=\varrho^{2-p}.
\end{equation}
This equation can be reduced to a first order ordinary differential
equation by using the standard trick of forcing the change of
variable $\varrho'=f(\varrho)$, thus
$\varrho''=f(\varrho)f'(\varrho)$. Then $f(\varrho)$ solves the
ordinary differential equation
\begin{equation*}
f'(\varrho)f(\varrho)=\frac{k}{\varrho}f^{2}(\varrho)-\varrho^{1-p},
\end{equation*}
which can be explicitly integrated if we make a further change of
variable by letting $f(\varrho)=\varrho^{k}g(\varrho)$. Then
\begin{equation*}
g(\varrho)g'(\varrho)=-\varrho^{1-p-2k},
\end{equation*}
and, since $2-p-2k=(2-p)(N+1)(p_c-p)/2(p-1)>0$, we find
$$
g(\varrho) = \left(\frac{2(K_0^{2-p-2k}-\varrho(r)^{2-p-2k})}{2-p-2k}\right)^{1/2} ,
$$
where $K_0$ is a generic constant. Coming back to the initial
variable $r$, \eqref{ODE} transforms to
\begin{equation}\label{ODE2}
\varrho'(r)=f(\varrho(r))=\varrho(r)^{k}\left(\frac{2(K_0^{2-p-2k}-\varrho(r)^{2-p-2k})}{2-p-2k}\right)^{1/2} .
\end{equation}
In other words, $\varrho$ is given in an implicit form through the
integral expression
$$
\left(\frac{2-p-2k}{2}\right)^{1/2}\int_0^{\varrho(r)}\frac{dz}{z^k(K_0^{2-p-2k}-z^{2-p-2k})^{1/2}}=r,
\quad r\in \left[ 0,\|u_0\|_{\infty} \right].
$$
Using the homogeneity of the integrand to scale $K_0$ out, we end up
with
$$
\left(\frac{(2-p-2k)K_0^{p}}{2}\right)^{1/2}\int_0^{\varrho(r)/K_0}\frac{dz}{z^k(1-z^{2-p-2k})^{1/2}}=r.
$$
A natural choice is then to take $\varrho(\|u_0\|_{\infty})=K_0$
which leads to
$$
\left(\frac{(2-p-2k)K_0^{p}}{2}\right)^{1/2}\int_0^{1}\frac{dz}{z^k(1-z^{2-p-2k})^{1/2}}=\|u_0\|_{\infty},
$$
that is,
\begin{equation}\label{est.C}
K_0=\kappa\ \|u_0\|_{\infty}^{2/p}
\end{equation}
for some positive constant $\kappa$ depending only on $N$, $p$, and
$q$. We also deduce from \eqref{ODE2} that $\varrho'(r)\leq
C\varrho(r)^{k}K_0^{(2-p-2k)/2}$, hence, since $k<1$ and
$\varrho(0)=0$, we find
\begin{equation}\label{est.rho}
\varrho(r)\leq CK_0^{(2-p-2k)/2(1-k)} r^{1/(1-k)}, \quad
r\in\left[ 0,\|u_0\|_{\infty} \right].
\end{equation}
We may now proceed along the lines of
Section~\ref{subsec.grad.est1}. Since $R_1=1$ by \eqref{ODE}, it
follows from \eqref{R} and \eqref{ODE2} that
\begin{equation}\label{generalR}
R=2(p-1)w^{(p+2)/2}+2(q-1)\varrho(u)^{q-1+k}\left(\frac{2(K_0^{2-p-2k}-\varrho(u)^{2-p-2k})}{2-p-2k}\right)^{1/2} w^{(q+2)/2}.
\end{equation}
If $q\geq1$ we omit the term coming from $R_2$ as it is non-negative
and deduce from \eqref{generalR} and the comparison principle that 
\begin{equation}\label{grad.int}
|\nabla u(t,x)|\leq\varrho(u(t,x))(p(p-1)t)^{-1/p}, \quad (t,x)\in Q_{\infty}.
\end{equation}
We plug the estimates \eqref{est.C} and \eqref{est.rho} into
\eqref{grad.int} and obtain the following estimate
\begin{equation*}
|\nabla u(t,x)|u(t,x)^{-1/(1-k)}\leq
C\|u_0\|_{\infty}^{(2-p-2k)/p(1-k)} t^{-1/p},
\end{equation*}
whence
\begin{equation}\label{grad.est8}
|\nabla u^{-k/(1-k)}(t,x)|\leq C\|u_0\|_{\infty}^{(2-p-2k)/p(1-k)} t^{-1/p},
\end{equation}
if $k\neq0$ (that is, $p\neq p_{sc}$) and
\begin{equation}\label{grad.est9}
|\nabla\log u(t,x)|\leq C\|u_0\|_{\infty}^{(2-p)/p} t^{-1/p},
\end{equation}
if $k=0$, that is, $p=p_{sc}=2(N+1)/(N+3)$.

We are left with the case $q\in[1-k,1)$ (which is only possible if
$k>0$, thus $p>p_{sc}$). In this case, starting from
\eqref{generalR}, we use the monotonicity of $\varrho$, the identity
\eqref{est.C} and \eqref{est.rho}, and compensate the negative term
coming from $R_2$ in the following way:
\begin{equation*}
\begin{split}
R&\geq2(p-1)w^{(p+2)/2}-2(1-q)CK_0^{(2-p-2k)/2}\varrho(\|u_0\|_{\infty})^{q-1+k} w^{(q+2)/2}\\
&\geq2(p-1)w^{(p+2)/2}-C\|u_0\|_{\infty}^{((2-p-2k)+2(q-1+k))/p} w^{(q+2)/2} \\
&=2(p-1)w^{(p+2)/2}-C\|u_0\|_{\infty}^{(2q-p)/p}w^{(q+2)/2}.
\end{split}
\end{equation*}
Arguing as in Section~\ref{subsec.grad.est1}, we conclude that
\begin{equation*}
|\nabla u(t,x)|\leq
C\varrho(u(t,x))\left(\|u_0\|_{\infty}^{(2q-p)/p(p-q)}+t^{-1/p}\right),\quad (t,x)\in Q_\infty.
\end{equation*}
Using again the estimates \eqref{est.C} and \eqref{est.rho}, we
arrive to our final estimate
\begin{equation}\label{grad.est10}
\left| \nabla u^{-k/(1-k)}(t,x) \right|\leq
C\|u_0\|_{\infty}^{(2-p-2k)/p(1-k)}\left(\|u_0\|_{\infty}^{(2q-p)/p(p-q)}+t^{-1/p}\right)
\end{equation}
for $(t,x)\in Q_\infty$.

\subsection{Gradient estimates for the singular diffusion equation \eqref{eq:diffusion}}
\label{subsec.grad.SD}

A careful look at the proofs of the gradient estimates
\eqref{grad.est1}, \eqref{grad.est6}, \eqref{grad.est8}, and
\eqref{grad.est9} reveals that the contribution from the absorption
term is always omitted so that these estimates are also true for
solutions to the singular diffusion equation \eqref{eq:diffusion}
with initial data satisfying \eqref{assump}. Since these gradient
estimates seem to have been unnoticed before, we provide here a
precise statement.

\begin{theorem}\label{th:graddiff}
Consider a function $u_0$ satisfying \eqref{assump} and let
$\Phi$ be the solution to \eqref{eq:diffusion} with initial
condition $u_0$. Then:
\begin{enumerate}[(i)]
\item For $p\in (p_c,2)$, we have
$$ \left|\nabla \Phi^{-(2-p)/p}(t,x)\right|\leq\left(\frac{2-p}{p}\right)^{(p-1)/p}\eta^{1/p}t^{-1/p},
\ (t,x)\in Q_\infty.
$$
\item For $p=p_c$, we have
$$
\left| \nabla \Phi^{-1/N}(t,x) \right|\leq C \left(
\log\left(\frac{e \|u_0\|_\infty}{\Phi(t,x)} \right)
\right)^{1/p_c}t^{-1/p_c}, \ (t,x)\in Q_\infty.
$$
\item For $p\in (p_{sc},p_c)$, we have $k\in (0,1)$ and
$$
\left| \nabla \Phi^{-k/(1-k)}(t,x) \right|\leq C\|u_0\|_{\infty}^{(2-p-2k)/p(1-k)} t^{-1/p}
$$
for $(t,x)\in Q_\infty$ such that $\Phi(t,x)>0$.
\item For $p=p_{sc}$, we have $k=0$ and
$$
|\nabla\log \Phi(t,x)|\leq C\|u_0\|_{\infty}^{(2-p)/p} t^{-1/p}
$$
for $(t,x)\in Q_\infty$ such that $\Phi(t,x)>0$.
\item For $p\in (1,p_{sc})$, we have $k<0$ and
$$
\left| \nabla \Phi^{|k|/(1+|k|)}(t,x) \right|\leq C\|u_0\|_{\infty}^{(2-p-2k)/p(1-k)} t^{-1/p}, \ (t,x)\in Q_\infty.
$$
\end{enumerate}
\end{theorem}

\subsection{A gradient estimate coming from the
Hamilton-Jacobi term}\label{subsec.grad.HJ}

Apart from the previous gradient estimates, which result either from
the sole diffusion or are the outcome of the competition between the
two terms, we can prove another one which is an extension of a known
result for the non-diffusive Hamilton-Jacobi equation. We assume
that $p>p_{sc}=2(N+1)/(N+3)$, although in the applications we will
only need the range $p\geq p_c$.

\medskip

\noindent \textbf{Case 1: $q<1$}. As in \cite{GGK03}, take
$\varphi(r)=\|u_0\|_{\infty}-r^2$ directly in \eqref{R1} and
\eqref{R2}. Then $v=(\|u_0\|_{\infty}-u)^{1/2}$, and
\begin{equation*}
R_2=-2^{q-1}v^{q-2}, \quad R_1=2^{p-2}(1+k)v^{p-4}\,.
\end{equation*}
Since we are in the range $q<1$ and $p>p_{sc}$, we notice that
$R_1>0$ and we can forget about the effect of this term. We deduce
that
\begin{equation*}
Lw:=\partial_{t}w-Aw-B\cdot\nabla
w+2^{q}(1-q)(\|u_0\|_{\infty}-u)^{(q-2)/2} w^{(q+2)/2}\leq 0\,.
\end{equation*}
We then notice that the function
$W(t)=K\|u_0\|_{\infty}^{(2-q)/q}t^{-2/q}$, with a suitable choice
of $K$, is a supersolution for the operator $L$, since
\begin{equation*}
\begin{split}
LW(t)&=\left[2^q(1-q)K^{(q+2)/2}\left(\|u_0\|_{\infty}-u\right)^{(q-2)/2}\|u_0\|_{\infty}^{((2-q)/q)+((2-q)/2)} \right. \\
& \hspace{3cm} - \left. \frac{2}{q}K\|u_0\|_{\infty}^{(2-q)/q}\right] t^{-(q+2)/q}\\
&\geq2K\left[2^{q-1}(1-q)K^{q/2}-\frac{1}{q}\right]\|u_0\|_{\infty}^{(2-q)/q}t^{-(q+2)/q}\geq 0
\end{split}
\end{equation*}
as soon as we choose $K^{q/2}=2^{1-q}(1-q^2)$. By the comparison
principle, we find that
\begin{equation*}
\left|\nabla(\|u_0\|_{\infty}-u(t,x))^{1/2}\right|\leq
C\|u_0\|_{\infty}^{(2-q)/2q} t^{-1/q}.
\end{equation*}
Noticing that
\begin{equation*}
2\left|\nabla(\|u_0\|_{\infty}-u(t,x))^{1/2}\right|=(\|u_0\|_{\infty}-u(t,x))^{-1/2} |\nabla
u(t,x)|\geq \|u_0\|_{\infty}^{-1/2} |\nabla u(t,x)|,
\end{equation*}
we conclude that
\begin{equation*}
|\nabla u(t,x)|\leq C\|u_0\|_{\infty}^{1/q}t^{-1/q},
\ (t,x)\in Q_{\infty}.
\end{equation*}

\medskip

\noindent \textbf{Case 2: $q>1$}. In this case, let us take
$\varrho(u)=u^{1/q}$ in \eqref{altR1} and \eqref{altR2}, as in
\cite{BL99}. We calculate
\begin{equation*}
R_1=\frac{k+q-1}{q^2}u^{(p-2q)/q}>0, \quad R_2=\frac{1}{q}.
\end{equation*}
Since we want only an estimate coming from the absorption term, we
omit $R_1$ and we have
\begin{equation*}
Lw:=\partial_{t}w-Aw-B\cdot\nabla
w+\frac{2(q-1)}{q}w^{(2+q)/2}\leq 0.
\end{equation*}
We then notice that the function $W(t)=[(q-1)t]^{-2/q}$ is a
supersolution for the operator $L$. By the comparison principle, we
find that
\begin{equation*}
|\nabla u(t,x)|\leq\varrho(u(t,x))\left[\frac{1}{(q-1)t}\right]^{1/q},
\end{equation*}
or equivalently
\begin{equation*}
\left|\nabla u^{(q-1)/q}(t,x)\right|\leq\frac{1}{q}(q-1)^{(q-1)/q}t^{-1/q},
\quad (t,x)\in Q_{\infty}.
\end{equation*}

\begin{remark}\label{rem:4} There is no gradient estimate produced by
the Hamilton-Jacobi term for $q=1$, since its contribution vanishes
in \eqref{R}. This is in fact due to the lack of strict convexity
(or concavity) of the euclidean norm.
\end{remark}

\section{Decay estimates for integrable initial
data}\label{sec.decay}

We devote this section to the proof of Theorem~\ref{th:warmup}.
These decay rates will be improved in Section~\ref{se:idrae} for
$p>p_c$ and initial data which decay at infinity more rapidly than
what is required by mere integrability.

\begin{proposition}\label{prop.decay}
Let $u$ be a solution to \eqref{eq1}-\eqref{a2} with an initial
condition $u_0$ satisfying \eqref{assump}. The following decay
estimates hold:

\noindent (i) If $p>p_c$ and $q>q_{\star}=p-N/(N+1)$, we have
\begin{equation}\label{decay1.big}
\|u(t)\|_{\infty}\leq C\|u_{0}\|_{1}^{p\eta}t^{-N\eta}, \ t>0,
\end{equation}
where $\eta=1/[N(p-2)+p]$.

\noindent (ii) If $p>p_c$ and $q\in(N/(N+1),q_{\star}]$, we have
\begin{equation}\label{decay1.small}
\|u(t)\|_{\infty}\leq C\|u_0\|_{1}^{q\xi}t^{-N\xi}, \ t>0,
\end{equation}
where $\xi=1/[q(N+1)-N]$.
\end{proposition}

\begin{proof}
Denoting the solution to \eqref{eq:diffusion} with initial condition
$u_0$ by $\Phi$, the comparison principle guarantees that
$u\le\Phi$ in $Q_\infty$ and \eqref{decay1.big} readily follows
from \cite[Theorem~3]{HV81}. Next, the proof of \eqref{decay1.small}
for $q>1$ and $q\in (N/(N+1),1)$ relies on \eqref{grad.estHJ2} and
\eqref{grad.estHJ}, respectively, and is the same as that of
\cite[Proposition~1.4]{BtL08} and \cite[Theorem~1]{BLS01} to which
we refer. For $q=1$ we reproduce \emph{verbatim} the proof in
\cite[Section~3]{BLS01}.
\end{proof}

Since \eqref{eq1} is an autonomous equation, a simple consequence of
Proposition~\ref{prop.decay} is the following:

\begin{corollary}
Let $u$ be a solution of \eqref{eq1}-\eqref{a2} with an initial
condition $u_0$ satisfying \eqref{assump}. For $p\in(p_c,2)$ and
$q\in (N/(N+1),q_\star]$, we have \beqn \|u(t)\|_\infty \le C\
\|u(s)\|_1^{q\xi}\ (t-s)^{-N\xi}\,, \quad 0\le s < t. \label{spirou}
\eeqn For $q>q_{\star}$ we have \beqn \|u(t)\|_\infty \le C\
\|u(s)\|_1^{p\eta}\ (t-s)^{-N\eta}\,, \quad 0\le s < t.
\label{spirou2} \eeqn
\end{corollary}

We next turn to the case $p=p_c$ and first establish that the
solutions to the singular diffusion equation \eqref{eq:diffusion}
with non-negative integrable initial data decay exponentially for
large times. Though this property is expected, a proof does not seem
to be available in the literature.

\begin{proposition}\label{prop:expdecsd}
Consider a function $u_0$ satisfying \eqref{assump} and let $\Phi$ be the solution to \eqref{eq:diffusion}  with initial condition $u_0$ and $p=p_c$. Then
\beqn
\|\Phi(t)\|_\infty \le C'\ \|u_0\|_\infty\ e^{-C t /\|u_0\|_1^{2/(N+1)}}\,, \quad t\ge 0\,.
\label{pim}
\eeqn
\end{proposition}

\begin{proof}
By Theorem~\ref{th:graddiff}, we have
\bean
|\nabla\Phi(t,x)| & = & N\ \Phi(t,x)^{(N+1)/N}\ \left| \nabla \Phi^{-1/N}(t,x) \right| \\
& \le & C\ \Phi(t,x)^{(N+1)/N}\ \left( \log\left( \frac{e \|u_0\|_\infty}{\Phi(t,x)} \right) \right)^{1/p_c}\ t^{-1/p_c} \\
|\nabla\Phi(t,x)|^{p_c} & \le & C\ \Phi(t,x)^2\ \left( \log\left(
\frac{e^{3/2} \|u_0\|_\infty}{\Phi(t,x)} \right) \right)\ t^{-1}\,.
\eean Noticing that the function $z\mapsto z^2 \log\left(e^{3/2}
\|u_0\|_\infty /z \right)$ is non-decreasing in $[0,\|u_0\|_\infty]$
and that $0\le \Phi \le \|u_0\|_\infty$ in $Q_\infty$, we conclude
that
$$
|\nabla\Phi(t,x)|^{p_c} \le C\ \|\Phi(t)\|_\infty^2\ \log\left( \frac{e^{3/2} \|u_0\|_\infty}{\|\Phi(t)\|_\infty} \right)\ t^{-1}
$$
for $(t,x)\in Q_\infty$, while the Gagliardo-Nirenberg inequality
\beqn \Vert w\Vert_\infty \le C\
\Vert\nabla w\Vert_\infty^{N/(N+1)}\ \Vert w\Vert_1^{1/(N+1)}
\;\;\mbox{ for }\;\; w\in L^1(\RR^N)\cap
W^{1,\infty}(\RR^N)\,,\label{GN} \eeqn
ensures that
$$
\|\Phi(t)\|_\infty^2 \le C\ \|\nabla\Phi(t)\|_\infty^{p_c}\ \|\Phi(t)\|_1^{p_c/N}\,, \quad t>0\,.
$$
Combining the above two inequalities with the conservation of mass $\|\Phi(t)\|_1=\|u_0\|_1$ \cite[Theorem~2]{HV81}, we end up with
\bean
\|\Phi(t)\|_\infty^2 & \le & C \|\Phi(t)\|_\infty^2\ \log\left( \frac{e^{3/2} \|u_0\|_\infty}{\|\Phi(t)\|_\infty} \right)\ \frac{\|u_0\|_1^{p_c/N}}{t} \\
e^{C t / \|u_0\|_1^{p_c/N}} & \le & \frac{e^{3/2} \|u_0\|_\infty}{\|\Phi(t)\|_\infty}\,,
\eean
from which \eqref{pim} follows.
\end{proof}

\begin{proof}[Proof of Theorem~\ref{th:warmup}]
The estimates \eqref{wu1} and \eqref{wu2} are proved in
Proposition~\ref{prop.decay}. The exponential decay \eqref{wu3}
follows from Proposition~\ref{prop:expdecsd} and the comparison
principle when $p=p_c$ while it is proved as in
\cite[Theorem~2]{BLS01} for $p>p_c$ and $q=N/(N+1)$, the main tool
of the proof being the gradient estimate \eqref{grad.estHJ}. For
$p\ge p_c$ and $q\in (0,N/(N+1))$, the finite time extinction
\eqref{wu4} is a feature of the absorption term and is also a
consequence of \eqref{grad.estHJ}. We refer to
\cite[Theorem~1]{BLS01} or \cite[Theorem~3.1]{HJbook} for a proof.
Finally, the extinction for $p\in (1,p_c)$ follows by comparison
with the singular diffusion equation \eqref{eq:diffusion} for which
finite time extinction is known to occur for initial data in
$L^r(\RR^N)$ with suitable $r$ \cite{BIV,HV81,VazquezSmoothing},
noting that $L^1\cap L^{\infty}\subset L^r$ for any
$r\in(1,\infty)$.
\end{proof}

\section{Large time behavior of $\|u(t)\|_{1}$}\label{sec.large}

In this section we study the possible values of the limit as
$t\to\infty$ of the $L^1$-norm of solutions $u$ to \eqref{eq1}-\eqref{a2} with
initial data $u_0$ satisfying \eqref{assump}. The case $p\in
(1,p_c)$ being obvious as $u$ vanishes identically after a finite
time by Theorem~\ref{th:warmup}, we assume in this section that
$p\ge p_c$ and first state the time monotonicity of the $L^1$-norm
of $u$ \beqn \|u(t)\|_1 \le \|u(s)\|_1 \le \|u_0\|_1\,, \qquad
t>s\ge 0\,, \label{gaston} \eeqn which follows by  construction of
the solution, see \eqref{decaymass} below. This last inequality can
actually be improved to an equality for $p\ge p_c$ as we shall see
now.

\begin{proposition}\label{le:u1}
If $p\in (p_c,2)$, $q\in [p/2,\infty)$, and $u_0$ satisfies
\eqref{assump}, then
\begin{equation}\label{u1}
\|u(t)\|_1 + \int_0^{t}\int|\nabla u(s,x)|^q\ dxds = \|u_0\|_1\,,
\qquad t\ge0.
\end{equation}
\end{proposition}

\begin{remark}\label{re:cm} Let us point out here that this
result is not obvious as it is clearly false for the singular
diffusion equation \eqref{eq:diffusion} for $p<p_c$ for which we
have extinction in finite time. Therefore, it may only hold true for
$p\ge p_c$ and we refer to \cite[Theorem~2]{HV81} for a proof for
\eqref{eq:diffusion}. The proof of Proposition~\ref{le:u1} given
below for $p>p_c$ (and $q\ge p/2$) is however of a completely
different nature, relying on the gradient estimates
\eqref{grad.est1} and \eqref{grad.est2}, and provides an alternative
proof of the mass conservation for \eqref{eq:diffusion} for $p>p_c$.
The case $p=p_c$ will be considered in the next proposition, the
proof relying on arguments from \cite{HV81}.
\end{remark}

\begin{proof}
Let $\vartheta$ be a non-negative and smooth compactly supported
function in $\real^N$ such that $0\le\vartheta\le 1$,
$\vartheta(x)=1$ for $x\in B_1(0)$ and $\vartheta(x)=0$ for
$x\in\real^N\setminus B_2(0)$. For $R>1$ and $x\in\real^N$, we
define $\vartheta_R(x):=\vartheta(x/R)$. Since $p/(2-p)>1$, the
function $\vartheta_R^{p/(2-p)}$ is a non-negative compactly
supported $C^1$-smooth function and it follows from \eqref{weakform}
that, for $t>0$,
\begin{equation}\label{u2}
\begin{split}
I_R(t) & :=\int\vartheta_R^{p/(2-p)}\ u(t)\ dx + \int_0^t \int \vartheta_R^{p/(2-p)}\ |\nabla u(s)|^q\ dxds \\
&=\int\vartheta_R^{p/(2-p)}\ u_0\ dx - \int_0^t \int \nabla\left(
\vartheta_R^{p/(2-p)} \right)\cdot \left( |\nabla u|^{p-2}\ \nabla u
\right)(s)\ dxds\,.
\end{split}
\end{equation}
On the one hand, since $u_0\in L^1(\RR^N)$ and
$\vartheta_R^{p/(2-p)}\longrightarrow 1$ as $R\to\infty$ with
$\left| \vartheta_R^{p/(2-p)} \right| \le 1$, the Lebesgue dominated
convergence theorem guarantees that \beqn \lim_{R\to\infty} \int
\vartheta_R(x)^{p/(2-p)}\ u_0(x)\ dx = \|u_0\|_1\,. \label{u3} \eeqn
On the other hand, since $p>p_c$ and $q\ge p/2$, $u$ satisfies the
gradient estimate
$$
\left| \nabla u^{(p-2)/p}(s,x) \right|\le C(u_0)\ \left( 1+s^{-1/p}
\right)\,, \qquad (s,x)\in Q_\infty\,,
$$
by \eqref{grad.est1} and \eqref{grad.est2}. Since $\left| \nabla
u \right| = (p/(2-p))\ u^{2/p}\ \left| \nabla
u^{(p-2)/p} \right|$ and $2(p-1)<p$, we infer from the previous
gradient estimate and H\"older's and Young's inequalities that \bear
& & \left| \int_0^t \int \nabla\left( \vartheta_R^{p/(2-p)} \right)(x)\cdot \left( |\nabla u|^{p-2}\ \nabla u \right)(s,x)\ dxds \right| \nonumber\\
& \le & C\ \int_0^t \int \vartheta_R(x)^{2(p-1)/(2-p)}\ |\nabla\vartheta_R(x)|\ u(s,x)^{2(p-1)/p}\ \left| \nabla u^{(p-2)/p}(s,x) \right|^{p-1}\ dxds \nonumber\\
& \le & C(u_0)\ \int_0^t \left( 1+s^{-(p-1)/p} \right)\ \|\nabla\vartheta_R\|_{p/(2-p)}\ \left( \int \vartheta_R(x)^{p/(2-p)}\ u(s,x)\ dx\right)^{2(p-1)/p}\ ds \nonumber\\
& \le & C(u_0,\vartheta)\ R^{-(N+1)(p-p_c)/p}\ \int_0^t \left( 1+s^{-(p-1)/p} \right)\ I_R(s)^{2(p-1)/p}\ ds \nonumber\\
& \le & C(u_0,\vartheta)\ R^{-(N+1)(p-p_c)/p}\ \int_0^t \left(
1+s^{-(p-1)/p} \right)\ (1+I_R(s))\ ds \,. \label{u4} \eear It now
follows from \eqref{u2}, \eqref{u4}, and Gronwall's lemma that \beqn
I_R(t) \le \left( 1+ \|u_0\|_1 \right)\ \exp{\left\{
C(u_0,\vartheta)\ R^{-(N+1)(p-p_c)/p}\ (t+t^{1/p}) \right\}} - 1\,,
\qquad t\ge 0\,. \label{u5} \eeqn Since
$\vartheta_R^{p/(2-p)}\longrightarrow 1$ as $R\to\infty$ and the
right-hand side of \eqref{u5} is bounded independently of $R>1$, we
deduce from \eqref{u5} and Fatou's lemma that $u(t)\in L^1(\RR^N)$
and $|\nabla u|^q\in L^1((0,t)\times\RR^N)$ for every $t>0$. We are
then in a position to apply once more the Lebesgue dominated
convergence theorem to conclude that \beqn \lim_{R\to\infty} I_R(t)
= \|u(t)\|_1 + \int_0^t \int |\nabla u(s,x)|^q\ dxds\,, \qquad
t>0\,, \label{u6} \eeqn while \eqref{u4}, \eqref{u5}, and the
assumption $p>p_c$ ensure that \beqn \lim_{R\to\infty} \int_0^t \int
\nabla\left( \vartheta_R^{p/(2-p)} \right)(x)\cdot \left( |\nabla
u|^{p-2}\ \nabla u \right)(s,x)\ dxds = 0\,, \qquad t>0\,.
\label{u7} \eeqn We may then pass to the limit as $R\to\infty$ in
\eqref{u2} and use \eqref{u3}, \eqref{u6}, and \eqref{u7} to obtain
\eqref{u1}.
\end{proof}

We complete now the panorama with the corresponding result for
$p=p_c>1$, which requires $N\ge 2$.

\begin{proposition}\label{le:u1b}
If $p=p_c$, $q>0$, and $u_0$ satisfies \eqref{assump} along with
\beqn
u_0(x) \le C_0\ |x|^{-N}\,, \quad x\in\RR^N\,,
\label{u0b}
\eeqn
for some $C_0>0$, then
\begin{equation}\label{u1b}
\|u(t)\|_1 + \int_0^{t}\int|\nabla u(s,x)|^q\ dxds = \|u_0\|_1\,,
\qquad t\ge0.
\end{equation}
\end{proposition}

\begin{proof}
A straightforward computation shows that $(t,x)\mapsto C_0\
|x|^{-N}$ is a supersolution to \eqref{eq1} in
$(0,\infty)\times\RR^N\setminus\{0\}$ and we infer from \eqref{u0b}
and the comparison principle that \beqn u(t,x) \le C_0\ |x|^{-N}\,,
\quad (t,x)\in [0,\infty)\times\RR^N\setminus\{0\}\,. \label{u2b}
\eeqn Next, let $\vartheta$ be a non-negative and smooth compactly
supported function in $\real^N$ such that $0\le\vartheta\le 1$,
$\vartheta(x)=1$ for $x\in B_1(0)$, and $\vartheta(x)=0$ for
$x\in\real^N\setminus B_2(0)$. For $R>1$ and $x\in\real^N$, we
define $\vartheta_R(x):=\vartheta(x/R)$. We multiply \eqref{eq1} by
$(1-\vartheta_R)^{p_c}\ u$, integrate over $\RR^N$, and use Young's
inequality to obtain
\begin{equation*}
\begin{split}
\frac{1}{2}\ & \frac{d}{dt} \int (1-\vartheta_R)^{p_c}\ u^2\ dx \le - \int \nabla\left( (1-\vartheta_R)^{p_c}\ u \right)\cdot |\nabla u|^{p_c-2}\ \nabla u\ dx \\
& \le - \int (1-\vartheta_R)^{p_c}\ |\nabla u|^{p_c}\ dx + {p_c}\ \int |\nabla\vartheta_R|\ \left( (1-\vartheta_R)\ |\nabla u| \right)^{p_c-1}\ u\ dx \\
& \le -(2-p_c)\ \int (1-\vartheta_R)^{p_c}\ |\nabla u|^{p_c}\ dx  + \int |\nabla\vartheta_R|^{p_c}\ u^{p_c}\ dx\,.
\end{split}
\end{equation*}
Integrating with respect to time over $(0,t)$ and using the properties of $\vartheta_R$, \eqref{u0b}, and \eqref{u2b} give
\begin{equation*}
\begin{split}
(2-p_c)\ & \int_0^t \int_{\{|x|\ge 2R\}} |\nabla u|^{p_c}\ dxds \le (2-p_c)\ \int_0^t \int (1-\vartheta_R)^{p_c}\ |\nabla u|^{p_c}\ dxds \\
& \le  \frac{1}{2}\ \int (1-\vartheta_R)^{p_c}\ u_0^2\ dx + \frac{1}{R^{p_c}}\ \int_0^t \int \left| \nabla\vartheta\left( \frac{x}{R} \right) \right|^{p_c}\ u^{p_c}\ dxds \\
& \le \frac{C_0}{2}\ \int_{\{|x|\ge R\}} \frac{u_0(x)}{|x|^N}\ dx + \frac{C_0^{p_c-1} \|\nabla\vartheta\|_\infty^{p_c}}{R^{p_c}}\ \int_0^t \int_{\{|x|\ge R\}} \frac{u(s,x)}{|x|^{N(p_c-1)}}\ dxds \\
& \le \frac{C_0}{2R^N}\ \int_{\{|x|\ge R\}} u_0(x)\ dx +
\frac{C_0^{p_c-1} \|\nabla\vartheta\|_\infty^{p_c}}{R^N}\ \int_0^t
\int_{\{|x|\ge R\}} u(s,x)\ dxds\,,
\end{split}
\end{equation*}
whence
\beqn
\int_0^t \int_{\{|x|\ge 2R\}} |\nabla u|^{p_c}\ dxds \le \frac{C(\vartheta,u_0)}{R^N}\ \omega(t,R)\,,
\label{u3b}
\eeqn
with
$$
\omega(t,R) := \int_{\{|x|\ge R\}} u_0(x)\ dx + \int_0^t \int_{\{|x|\ge R\}} u(s,x)\ dxds\,.
$$
Now, owing to \eqref{u3b} and H\"older's inequality, we have
\bean
& & \left| \int_0^t \int \nabla\left( \vartheta_R^N \right) \cdot |\nabla u|^{p_c-2}\nabla u\ dxds \right| \\
& \le & N\ \left( \int_0^t \int_{\{|x|\ge R\}} |\nabla u|^{p_c}\ dxds  \right)^{(p_c-1)/p_c}\ \left( \int_0^t \int |\nabla\vartheta_R|^{p_c}\ dxds \right)^{1/p_c} \\
& \le & N\ \left[ \frac{2^N C(\vartheta,u_0)}{R^N}\ \omega\left( t , \frac{R}{2} \right) \right]^{(N-1)/2N}\ \|\nabla\vartheta\|_{p_c}\ t^{1/p_c}\ R^{(N-p_c)/p_c} \\
& \le & C(\vartheta,u_0)\ t^{1/p_c}\ \omega\left( t , \frac{R}{2}
\right)^{(N-1)/2N}\,. \eean Since $u\in L^\infty(0,t;L^1(\RR^N))$ by
\eqref{decaymass} and $u_0\in L^1(\RR^N)$, it readily follows from
the Lebesgue dominated convergence theorem that $\omega(t,R/2)\to 0$
as $R\to\infty$. We have thus proved that \eqref{u7} also holds true
for $p=p_c$ (since $p_c/(2-p_c)=N$) and we can proceed as in the end
of the proof of Proposition~\ref{le:u1} to complete the proof.
\end{proof}

We prove now a first result concerning non-extinction in finite time
in the range $q>p/2$. Apart from the interest by itself, this result
is also a technical step in the proof of the next estimates.

\begin{proposition}\label{nonextinction}
Let $p\ge p_c$, $q\in(p/2,\infty)$, and an initial condition $u_0$
satisfying \eqref{assump} as well as \eqref{u0b} if $p=p_c$. Then
the solution of \eqref{eq1}-\eqref{a2} cannot vanish in finite time.
\end{proposition}

\begin{proof}
We borrow some ideas from \cite[Lemma 4.1]{ATU04}. Assume for
contradiction that there exists $T\in (0,\infty)$ such that
$u(T)\equiv0$ and $\|u(t)\|_{1}>0$ for any $t\in[0,T)$. For
$\theta\in(0,1)$ to be specified later, define
\begin{equation}\label{energ.teta}
E_{\theta}(t)=\int u(t,x)^{1+\theta}\,dx , \quad t\ge 0.
\end{equation}
Let $\lambda>0$ (to be chosen later) and $Q\in(p/2,p)$ such that
$Q\leq q$. We use Proposition~\ref{le:u1} for $p>p_c$ or
Proposition~\ref{le:u1b} for $p=p_c$, \eqref{prunelle}, and
H\"{o}lder's inequality to get
\begin{equation*}
\begin{split}
\frac{d}{dt}\|u(t)\|_{1}&=-\int|\nabla u|^{q}\,dx\geq-\|\nabla
u_0\|_{\infty}^{q-Q}\int|\nabla u|^{Q} u^{-\lambda}
u^{\lambda}\,dx\\&\geq-C(u_0)\ \left(\int|\nabla
u|^{p}u^{-p\lambda/Q}\,dx\right)^{Q/p}\left(\int
u^{p\lambda/(p-Q)}\,dx\right)^{(p-Q)/p}.
\end{split}
\end{equation*}
We now choose $\lambda$ in order to find the derivative of
$E_{\theta}$ in the first factor in the right-hand side of the above
inequality. More specifically, by differentiating in
\eqref{energ.teta} and using \eqref{eq1}, we find
\begin{equation*}
\begin{split}
\frac{d}{dt}E_{\theta}(t)&=(1+\theta)\int
u(t,x)^{\theta}(\Delta_{p}u(t,x)-|\nabla u(t,x)|^q)\,dx\\
&\leq-\theta(1+\theta)\int u(t,x)^{\theta-1}|\nabla u(t,x)|^p\,dx,
\end{split}
\end{equation*}
hence, we choose $\lambda$ such that $p\lambda/Q=1-\theta>0$. The
inequality thus becomes
\begin{equation}
\frac{d}{dt}\|u(t)\|_1\geq-C(u_0,\theta)
\left(-\frac{d}{dt}E_{\theta}(t)\right)^{Q/p}\left(\int
u(t,x)^{Q(1-\theta)/(p-Q)}\,dx\right)^{(p-Q)/p}.
\end{equation}
We choose $\theta$ such that $Q(1-\theta)/(p-Q)=1$, that is
$\theta=(2Q-p)/Q\in(0,1)$. Using Young's inequality, we arrive to
the differential inequality
\begin{equation*}
\frac{d}{dt}\|u(t)\|_1\geq-C(u_0,\theta)
\left(-\frac{d}{dt}E_{\theta}(t)\right)^{Q/p}\|u(t)\|_{1}^{(p-Q)/p}
\geq\e\frac{d}{dt}E_{\theta}(t)-C(u_0,\theta,\e)\|u(t)\|_1,
\end{equation*}
for $\e>0$; we integrate it on $(t,T)$ and use the time
monotonicity \eqref{decaymass} of $\|u\|_1$ to get
\begin{equation*}
-\|u(t)\|_{1}+C(u_0,\theta,\e)(T-t)\|u(t)\|_1\geq-\|u(t)\|_1+C(u_0,\theta,\e)\int_{t}^{T}\|u(s)\|_1\,ds\geq-\e
E_{\theta}(t)\,,
\end{equation*}
whence
\beqn\label{PPP}
\liminf\limits_{t\to
T}\frac{E_{\theta}(t)}{\|u(t)\|_{1}}\geq\frac{1}{\e}\,.
\eeqn
But on the other hand, we notice that
\begin{equation*}
\frac{E_{\theta}(t)}{\|u(t)\|_{1}}\leq \|u(t)\|_{\infty}^{\theta}\to
0 \quad \hbox{as} \ t\to T,
\end{equation*}
which is a contradiction with \eqref{PPP}. Thus, there cannot be a
finite extinction time $T>0$.
\end{proof}
As a consequence of this non-extinction result, we are able to prove
that, for $p>p_c$ and $q>p/2$, the positivity set is the whole set
$Q_\infty$.

\begin{corollary}\label{cor:positivity}
If $p\ge p_c$, $q>p/2$, and $u_0$ satisfies \eqref{assump} as well
as \eqref{u0b} if $p=p_c$, then the solution of
\eqref{eq1}-\eqref{a2} is such that $u(t,x)>0$ for $(t,x)\in
Q_\infty$.
\end{corollary}

\begin{proof}
We first consider the case $p>p_c$. Let $t>0$ and $\delta\in (0,1)$.
We first recall that, since $p>p_c$ and $q>p/2$, we have \beqn
\left| \nabla (u+\delta)^{(p-2)/p}(t,x) \right| \le \phi(t) :=
C(u_0)\ \left( 1 + t^{-1/p} \right)\,, \quad x\in\RR^N\,,
\label{pos1} \eeqn by \eqref{grad.est1} and \eqref{grad.est2},
taking into account Remark~\ref{rem:1} and \eqref{grad.est.strong}.
Fix $x_0\in\RR^N$. For $x\in \RR^N$, we infer from \eqref{pos1} that
$$
(u(t,x_0)+\delta)^{(p-2)/p} \le (u(t,x)+\delta)^{(p-2)/p}
+ \phi(t)\ |x-x_0|\,.
$$
Multiplying the above inequality by $(u(t,x)+\delta)^{2/p}$ and
integrating with respect to $x$ over $B_r(x_0)$ for some $r>0$ to be
determined later give \bean
& & \left( \int_{B_r(x_0)} (u(t,x)+\delta)^{2/p}\ dx \right)\ (u(t,x_0)+\delta)^{(p-2)/p} \\
& \le & \int_{B_r(x_0)} \left[ u(t,x) + \delta + \phi(t)\
|x-x_0|\ (u(t,x)+\delta))^{2/p} \right]\ dx\,. \eean
Noting that
$$
M(r,\delta) := \int_{B_r(x_0)} (u(t,x)+\delta)\ dx \le
\left( \int_{B_r(x_0)} (u(t,x)+\delta)^{2/p}\ dx \right)^{p/2}\
|B_r(x_0)|^{(2-p)/2}
$$
by H\"older's inequality, we obtain
$$
|B_r(x_0)|^{(p-2)/p}\ M(r,\delta)^{2/p}\
(u(t,x_0)+\delta)^{(p-2)/p} \le M(r,\delta)\ \left( 1 + r\
\phi(t)\ \|u(t)+\delta \|_\infty^{(2-p)/p} \right)\,,
$$
$$
|B_r(x_0)|^{(p-2)/p}\ M(r,\delta)^{(2-p)/p} \le
(u(t,x_0)+\delta)^{(2-p)/p}\ \left( 1 + r\ \phi(t)\
\|u(t)+\delta \|_\infty^{(2-p)/p} \right)\,.
$$
Letting $\delta\to 0$, we end up with
$$
|B_r(x_0)|^{-1}\ M(r,0) \le u(t,x_0)\ \left( 1 + r\ \phi(t)\
\|u(t)\|_\infty^{(2-p)/p} \right)^{p/(2-p)}\,.
$$
Since $M(r,0)\to \|u(t)\|_1$ as $r\to\infty$ and $\|u(t)\|_1>0$ by
Proposition~\ref{nonextinction}, we may fix $r_0$ large enough such
that $M(r_0,0)>0$ and deduce from the above inequality with $r=r_0$
that
$$
0<|B_{r_0}(x_0)|^{-1}\ M(r_0,0) \le u(t,x_0)\ \left( 1 + r_0\
\phi(t)\ \|u(t)\|_\infty^{(2-p)/p} \right)\,,
$$
which shows the positivity of $u(t,x_0)$.

\medskip

Next, if $p=p_c$, $q\in (p_c/2,\infty)$, $\delta\in (0,1)$, and
$(t,x)\in Q_\infty$, it follows from \eqref{grad.est6},
\eqref{grad.est7}, and Remark~\ref{rem:3} that
$$
\left| \nabla (u+\delta)^{-1/N}(t,x) \right|\leq C(u_0) \ \left(
\log\left(\frac{e \|u_0\|_\infty}{u(t,x)+\delta} \right)
\right)^{1/p_c} \left( 1+t^{-1/p_c} \right).
$$
Fix $\theta\in (0,1/N)$. Then, owing to the boundedness of the
function $r\mapsto r^{(1-N\theta)/N} |\log{r}|^{1/p_c}$ for $r\in
[0,\|u_0\|_\infty+1]$, we have \bean
\left| \nabla (u+\delta)^{-\theta}(t,x) \right| & = & N\theta\ (u(t,x)+\delta)^{(1-N\theta)/N}\ \left| \nabla (u+\delta)^{-1/N}(t,x) \right| \\
& \le & C(\theta,u_0)\ \left( 1+t^{-1/p_c} \right), \eean for
$(t,x)\in Q_\infty$ and we may proceed as in the previous case to
establish the claimed positivity of $u$ in $Q_\infty$.
\end{proof}

We are now in a position to prove the two main results of this
section.

\begin{proposition}\label{positivity}
Let $u$ be a solution to \eqref{eq1}-\eqref{a2} with an initial
condition $u_0$ satisfying \eqref{assump} as well as \eqref{u0b} if
$p=p_c$. If $p\ge p_c$ and $q>q_{\star}$, then we have
$\lim\limits_{t\to\infty}\|u(t)\|_1>0$.
\end{proposition}

\begin{proof}
From Proposition~\ref{le:u1} (if $p>p_c$) and
Proposition~\ref{le:u1b} (if $p=p_c$), we have, for any $1\leq s\leq
t<\infty$:
\begin{equation}\label{cons}
\|u(s)\|_{1}=\|u(t)\|_{1}+\int_{s}^{t}\int\left(u(\tau,x)^{-1/q}|\nabla
u(\tau,x)|\right)^{q}u(\tau,x)\,dx\,d\tau\,.
\end{equation}
We want to use the gradient estimates \eqref{grad.est1},
\eqref{grad.est2}, \eqref{grad.est6}, and~\eqref{grad.est7}, and
thus split the proof into three cases.

\medskip

\noindent \textbf{Case 1: $p>p_c$ and $q\ge 1$}. In this case, by
using the gradient estimate \eqref{grad.est1}, together with the
decay estimate of the $L^{\infty}$-norm \eqref{decay1.big}, we
write, since $q>p/2$:
\begin{equation*}
\begin{split}
u(\tau,x)^{-1/q}\left|\nabla
u(\tau,x)\right|&=C\ u(\tau,x)^{(2q-p)/pq}\left|\nabla
u^{-(2-p)/p}(\tau,x)\right|\\
&\leq C\ \|u(\tau)\|_{\infty}^{(2q-p)/pq}\ \tau^{-1/p} \\
&\leq C\ \|u_0\|_{1}^{(2q-p)\eta/q}\ \tau^{-N\eta(2q-p)/pq-1/p},
\end{split}
\end{equation*}
hence
\begin{equation*}
\left( u^{-1/q}(\tau,x)\left|\nabla u(\tau,x)\right| \right)^q\leq
C(u_0)\tau^{-\eta/\xi}.
\end{equation*}
Plugging this inequality into \eqref{cons} and taking into account
that $\xi<\eta$, it follows that
\begin{equation*}
\begin{split}
\|u(s)\|_{1}&\leq\|u(t)\|_{1}+C(u_0)\int_{s}^{t}\|u(\tau)\|_{1}\tau^{-\eta/\xi}\,d\tau\\
&\leq\|u(t)\|_{1}+C(u_0)\|u(s)\|_{1}s^{-\eta(N+1)(q-q_{\star})},
\end{split}
\end{equation*}
where we have used the time monotonicity \eqref{decaymass} of the
$L^1$-norm. We can rewrite the last inequality as
\begin{equation}\label{last1}
\|u(t)\|_1\geq\|u(s)\|_1\left(1-C(u_0)s^{-\eta(N+1)(q-q_{\star})}\right).
\end{equation}
Using again that the exponent of $s$ in the right-hand side of
\eqref{last1} is negative, we realize that
$$\|u(t)\|_1\geq\frac{1}{2}\|u(s)\|_1\,, \quad t\ge s\,,$$ for $s$ large enough. Thus, using the non-extinction result of
Proposition~\ref{nonextinction}, we find that
$\lim\limits_{t\to\infty}\|u(t)\|_1>0$.

\medskip

\noindent \textbf{Case 2: $p>p_c$ and $q_{\star}<q<1$}. We use the
same ideas as above, but with slight changes since the gradient
estimate has now an extra term. Since \eqref{eq1} is autonomous, we
infer from \eqref{grad.est2} and \eqref{decay1.big} that
\begin{eqnarray*}
\left|\nabla u^{-(2-p)/p}(\tau,x)\right| & \leq & C\ \left( \left\| u\left( \frac{\tau}{2} \right) \right\|_\infty^{(2q-p)/p(p-q)}
+ \left( \frac{2}{\tau} \right)^{1/p} \right)\\
& \leq & C(u_0)\tau^{-1/p} \left(
1+\tau^{-\eta(N+1)(q-q_{\star})/(p-q)} \right)\leq
C(u_0)\tau^{-1/p},
\end{eqnarray*}
for any $\tau\geq1$. The proof then is the same as in Case~1 above.

\medskip

\noindent \textbf{Case 3: $p=p_c$ and $q>q_{\star}=p_c/2$}. To
estimate $u^{-1/q}\ |\nabla u|$, we use \eqref{wu3},
\eqref{grad.est6} (if $q\ge 1$) or \eqref{grad.est7} (if $q\in
(p_c/2,1)$), and the boundedness of the function $z\mapsto
z^{(2q-p_c)/2p_c q}\ \log(e\|u_0\|_\infty/z)$ in
$[0,\|u_0\|_\infty]$ to obtain, since $\tau\ge s\ge 1$, \bean
u(\tau,x)^{-1/q}\ |\nabla u(\tau,x)| & \le & C\ u(\tau,x)^{(2q-p_c)/p_c q}\ \left| \nabla u^{-1/N}(\tau,x) \right| \\
& \le & C(u_0)\ u(\tau,x)^{(2q-p_c)/p_c q}\ \left( \log\left( \frac{e \|u_0\|_\infty}{u(\tau,x)} \right) \right)^{1/p_c}\ \tau^{-1/p_c} \\
& \le & C(u_0)\ u(\tau,x)^{(2q-p_c)/2p_c q}\ \tau^{-1/p_c} \\
& \le & C'(u_0)\ e^{-C(u_0) \tau}\,. \eean This estimate,
\eqref{cons}, and the time monotonicity \eqref{decaymass} of the
$L^1$-norm lead us to
$$
\|u(s)\|_1 \le \|u(t)\|_1 + C'(u_0)\ \|u(s)\|_1\ e^{-C(u_0)s}\,,
\quad t\ge s\,,
$$
and we complete the proof as above with the help of
Proposition~\ref{nonextinction}.
\end{proof}

For the complementary case, things are different.

\begin{proposition}\label{zerolimit}
Let $p\in (1,2)$ and $q\in(0,q_{\star}]$. Then
$\lim\limits_{t\to\infty}\|u(t)\|_1=0$.
\end{proposition}

\begin{proof}
The proof follows that of \cite[Proposition~5.1]{BtL08}. For $t\ge
0$, we have
\begin{equation*}
\|u(t)\|_{1}+\int_0^t\int |\nabla u(s,x)|^q\,dxds\le\|u_0\|_1
\end{equation*}
by \eqref{decaymass}, hence $|\nabla u|^q\in
L^1((0,\infty)\times\RR^N)$. Therefore
\begin{equation}\label{omeg}
\omega(t):=\int_t^{\infty}\int |\nabla u(s,x)|^q\,dxds\to 0 \
\hbox{as} \ t\to\infty.
\end{equation}
Consider now a non-negative and smooth compactly supported function
$\vartheta$ such that $0\le\vartheta\le 1$, $\vartheta(x)=1$ for
$x\in B_1(0)$ and $\vartheta(x)=0$ for $x\in\real^N\setminus B_2(0)$
and define $\vartheta_{R}(x)=\vartheta(x/R)$ for $R>1$ and
$x\in\RR^N$. We multiply the equation \eqref{eq1} by
$1-\vartheta_{R}$ and integrate over $(t_1,t_2)\times\RR^N$ to
obtain
\begin{equation*}
\begin{split}
\int u(t_2,x)(1-\vartheta_R(x))\,dx &\leq\int
u(t_1,x)(1-\vartheta_{R}(x))\,dx \\
& +\int_{t_1}^{t_2}\int|\nabla u(s,x)|^{p-2}\nabla
u(s,x)\cdot\nabla\vartheta_R(x)\,dxds,
\end{split}
\end{equation*}
hence, taking into account the definition of $\vartheta_R$,
\begin{equation}\label{qqqq}
\int\limits_{|x|\geq2R}u(t_2,x)\,dx\leq\int\limits_{|x|\geq
R}u(t_1,x)\,dx+\frac{1}{R}\int_{t_1}^{t_2}\int|\nabla
u(s,x)|^{p-1}|\nabla\vartheta(x/R)|\,dxds.
\end{equation}
We now divide the proof into two cases.

\medskip

\noindent \textbf{Case 1: $p\ge p_c$, $q\in [N/(N+1), q_{\star}]$}.
Let us first consider the case where $q\in[p-1,q_{\star}]$ and
$q>N/(N+1)$. We apply H\"older's inequality to estimate
\begin{eqnarray*}
& & \frac{1}{R}\int_{t_1}^{t_2}\int |\nabla u(s,x)|^{p-1}|\nabla\vartheta(x/R)|\,dxds \\
& \leq & R^{(N(q-p+1)-q)/q} \|\nabla\vartheta\|_{q/(q-p+1)}
(t_2-t_1)^{(q-p+1)/q} \left(\int_{t_1}^{t_2} \int|\nabla
u(s,x)|^q\,dxds\right)^{(p-1)/q} \\
& \leq & C(\vartheta)R^{(N(q-p+1)-q)/q} (t_2-t_1)^{(q-p+1)/q}
\omega(t_1)^{(p-1)/q},
\end{eqnarray*}
hence, replacing in \eqref{qqqq} we obtain
\begin{eqnarray}
\|u(t_2)\|_{1}&=&\int\limits_{|x|<2R}u(t_2,x)\,dx+\int\limits_{|x|\ge2R}u(t_2,x)\,dx\nonumber\\
&\leq&
CR^{N}\|u(t_2)\|_{\infty}+C(\vartheta)R^{(N(q-p+1)-q)/q} (t_2-t_1)^{(q-p+1)/q} \omega(t_1)^{(p-1)/q} \nonumber\\
& & +\int\limits_{|x|\ge R}u(t_1,x)\,dx. \label{gaston}
\end{eqnarray}
Taking into account that $\|u(t_2)\|_{\infty}\leq C(u_0)
(t_2-t_1)^{-N\xi}$ by \eqref{decay1.small}, we optimize in $R$ in
the previous inequality. Choosing
\begin{equation*}
R=R(t_1,t_2):=\omega(t_1)^{(p-1)/(N(p-1)+q)} (t_2-t_1)^{(qN\xi+q-p+1)/(q+N(p-1))},
\end{equation*}
we obtain
\begin{equation*}
\begin{split}
\|u(t_2)\|_\infty\le & C(u_0,\vartheta)\ \omega(t_1)^{N(p-1)/(N(p-1)+q)}\ (t_2-t_1)^{qN(N+1)\xi(q-q_\star)/(N(p-1)+q)} \\
& + \int\limits_{|x|\ge R(t_1,t_2)}u(t_1,x)\ dx.
\end{split}
\end{equation*}
Noting that
$$
qN\xi+q-p+1 = \xi\ (q(N+1) (q-p+1) + N(p-1))> 0
$$
since $\xi>0$ and $q\ge p-1$, we may let $t_2\to\infty$ in the
previous estimate to obtain that $\|u(t_2)\|_\infty\to 0$ as
$t_2\to\infty$ when $q<q_\star$, and that
\begin{equation*}
\lim\limits_{t\to\infty}\|u(t)\|_1\leq C(u_0,\vartheta)
\omega(t_1)^{N(p-1)/(q+N(p-1))}\to 0 \ \hbox{as} \ t_1\to 0,
\end{equation*}
for $q=q_\star$.

\medskip

In the remaining case we can always fix $Q\ge q$ such that $Q\in
(p-1,q_{\star})$ and $Q>N/(N+1)$. Introducing
$$
\tilde{u}(t,x) := \|\nabla u_0\|_\infty^{-(Q-q)/(Q-p+1)}\ u\left(
\|\nabla u_0\|_\infty^{((2-p)(Q-q))/(Q-p+1)}\ t , x \right)\,,
\qquad (t,x)\in Q_\infty\,,
$$
we deduce from \eqref{eq1},\eqref{a2}, and \eqref{prunelle} that
\bean
\partial_t\tilde{u}(t,x) & = & \|\nabla u_0\|_\infty^{-((p-1)(Q-q))/(Q-p+1)}\ \partial_ t u\left( \|\nabla u_0\|_\infty^{((2-p)(Q-q))/(Q-p+1)}\ t , x \right) \\
& = & \|\nabla u_0\|_\infty^{-((p-1)(Q-q))/(Q-p+1)}\ (\Delta_p u - |\nabla u|^q)\left( \|\nabla u_0\|_\infty^{((2-p)(Q-q))/(Q-p+1)}\ t , x \right) \\
& \le & \Delta_p \tilde{u}(t,x) - \|\nabla u_0\|_\infty^{Q-q}\ |\nabla\tilde{u}(t,x)|^Q \left\|\nabla u\left( \|\nabla u_0\|_\infty^{((2-p)(Q-q))/(Q-p+1)}\ t \right) \right\|_\infty^{q-Q} \\
& \le & \Delta_p \tilde{u}(t,x) - |\nabla\tilde{u}(t,x)|^Q\,, \eean
with $\tilde{u}(0)=U_0:=\|\nabla u_0\|_\infty^{-(Q-q)/(Q-p+1)}\
u_0$. Denoting the solution to \eqref{eq1}-\eqref{a2} with $Q$
instead of $q$ and $U_0$ instead of $u_0$ by $U$, the comparison
principle entails that $\tilde{u}\le U$ in $Q_\infty$. According to
the choice of $Q$, we are in the situation of the previous case and
thus $\|U(t)\|_1\to 0$ as $t\to \infty$ and so do
$\|\tilde{u}(t)\|_1$ and $\|u(t)\|_1$.

\medskip

\noindent \textbf{Case 2: $p\ge p_c$ and $q\in(0,N/(N+1))$ or
$p<p_c$}. It is an obvious consequence of the extinction in finite
time established in Theorem~\ref{th:warmup}.
\end{proof}

\section{Improved decay rates and extinction}\label{se:idrae}

While the behavior of solutions $u$ to \eqref{eq1} depends strongly
on the values of $p$ and $q$ as depicted in Theorem~\ref{th:idrext},
it turns out that, as we shall see below, the proofs also vary with
these two parameters. Indeed, recalling the definition of $q_1$ in
\eqref{notations.cons}, finite time extinction will follow by the
comparison principle when either $p\in (1,p_c)$ or $p\ge p_c$ and
$q\in (0,q_1]$, while a differential inequality will be used for
$p>p_c$ and $q\in (q_1,p/2)$. A similar differential inequality will
actually allow us to prove the stated temporal decay rates for
$p>p_c$ and $q\in [p/2,q_\star)$. The particular case $p=p_c$ has to
be handled separately. Still, the proof of Theorem~\ref{th:idrext}
for $p\in (p_c,2)$ and $q\in (q_1,q_\star)$, $(p,q)\ne (p_c,p_c/2)$,
relies on the following preliminary result:

\begin{lemma}\label{le:v1}
Assume that $p\in (p_c,2)$, $q\in (p-1,q_\star)$, and consider $u_0$
satisfying \eqref{assump} and \beqn 0 \le u_0(x) \le K_0\
|x|^{-(p-q)/(q-p+1)}\,, \quad x\in\RR^N\,, \label{v1} \eeqn for some
$K_0>0$. Then, for $s\ge 0$ and $t>s$, we have \beqn \|u(t)\|_1 \le
C(u_0)\ \|u(t)\|_\infty^{\theta}\,, \label{v2} \eeqn with \beqn
\theta := (N+1)(q_\star-q)/(p-q)\,. \label{v4} \eeqn Assume further
that $q\in (q_1,q_\star)$. Then \beqn \|u(t)\|_1 \le C(u_0)\
\|u(s)\|_1^{q\xi \theta}\ (t-s)^{-N\xi \theta}\,, \label{v3} \eeqn
where $\xi$ is defined in \eqref{notations.cons}.
\end{lemma}

\begin{proof}
For $x\in\RR^N$, $x\ne 0$, we define
$$\Sigma_{p,q}(x) :=
|x|^{-(p-q)/(q-p+1)} \;\;\mbox{ and }\;\; A_0 := \frac{q-p+1}{p-q}\
\left( \frac{N(p-1)-q(N-1)}{q-p+1} \right)^{1/(q-p+1)}\,.
$$
An easy computation shows that, for any $A\ge A_0$, $A\
\Sigma_{p,q}$ is a classical (stationary) supersolution to
\eqref{eq1} in $\RR^N\setminus\{0\}$. Owing to \eqref{v1} $u_0 \le
A\ \Sigma_{p,q}$ for $A=\max{\{K_0,A_0\}}$ and the comparison
principle ensures that \beqn u(t,x) \le A\ \Sigma_{p,q}(x)\,, \qquad
(t,x)\in Q_\infty\,. \label{v6} \eeqn Since $q<q_\star$, it follows
from \eqref{v6} that, for $t>0$ and $R>0$, we have \bean
\|u(t)\|_1 & \le & \int_{B_R(0)} u(t,x)\ dx + \int_{\RR^N\setminus B_R(0)} u(t,x)\ dx \\
& \le & C\ R^N\ \|u(t)\|_\infty + C(u_0)\ \int_R^\infty r^{N-1-((p-q)/(q-p+1))}\ dr \\
& \le & C(u_0)\ \left( R^N\ \|u(t)\|_\infty +
R^{-(N+1)(q_\star-q)/(q-p+1)} \right)\,. \eean Choosing $R=\left(
\|u(t)\|_\infty + \delta \right)^{-(q-p+1)/(p-q)}$ for $\delta\in
(0,1)$, we obtain that
$$
\|u(t)\|_1 \le C\ \left( \|u(t)\|_\infty  + \delta
\right)^{\theta}\,,
$$
the parameter $\theta$ being defined in \eqref{v4}. Since $\theta>0$
and the above inequality is valid for all $\delta\in (0,1)$, we end
up with \eqref{v2} after letting $\delta\to 0$. We next combine
\eqref{spirou} and \eqref{v2} to deduce \eqref{v3}.
\end{proof}

\subsection{Improved decay}

In this subsection we prove the first part of Theorem~\ref{th:idrext}.

\begin{proof}[Proof of Theorem~\ref{th:idrext}~(i): $p\in(p_c,2)$ and $q\in (p/2,q_\star)$]
Consider $T>0$ and define
$$
m(T) := \sup_{t\in (0,T]}{\left\{ t^{(p-q)\theta/(2q-p)}\ \|u(t)\|_1
\right\}}\,,
$$
the parameter $\theta$ being defined in \eqref{v4}. Let $t\in
(0,T]$. Since $u_0$ satisfies \eqref{x1} and $q\in (q_1,q_\star)$,
we infer from \eqref{v3} with $s=t/2$ that \bean
t^{(p-q)\theta/(2q-p)}\ \|u(t)\|_1 & \le & C(u_0)\ \left\| u\left( \frac{t}{2} \right) \right\|_1^{q\xi \theta}\ t^{(p-q-N\xi(2q-p))\theta/(2q-p)}\\
& \le & C(u_0)\ \left\| u\left( \frac{t}{2} \right) \right\|_1^{q\xi \theta}\ \left( \frac{t}{2} \right)^{q(p-q)\xi\theta^2/(2q-p)} \\
& = & C(u_0)\ \left\{ \left( \frac{t}{2} \right)^{(p-q)\theta/(2q-p)}\ \left\| u\left( \frac{t}{2} \right) \right\|_1 \right\}^{q\xi \theta} \\
& \le & C(u_0)\ m(T)^{q\xi \theta}\,. \eean The above estimate being
valid for all $t\in (0,T]$, we conclude that $m(T)\le C(u_0)\
m(T)^{q\xi \theta}$, whence $m(T)\le C(u_0)$ since
$$
q\xi \theta = 1 - \frac{N\xi(2q-p)}{p-q} < 1\,.
$$
Since the constant $C(u_0)$ in the bound on $m(T)$ does not depend
on $T>0$, we have thus shown that \beqn \|u(t)\|_1\le C(u_0)\
t^{-(p-q)\theta/(2q-p)}\,, \qquad t>0\,.\label{v7} \eeqn Combining
\eqref{spirou} (with $s=t/2$) and \eqref{v7} gives
$$
\|u(t)\|_\infty\le C(u_0)\ t^{-(p-q)/(2q-p)}\,, \qquad t>0\,,
$$
and completes the proof of \eqref{x2}.
\end{proof}

\subsection{Exponential decay}

In this subsection we prove the second part of Theorem~\ref{th:idrext}, which illustrates the role of branching point that
our new (and initially unexpected) critical exponent $q=p/2$ plays
on the large time behavior of solutions to \eqref{eq1}.

\begin{proof}[Proof of Theorem~\ref{th:idrext}~(ii): $p\in (p_c,2)$ and $q=p/2$]
In that case, the parameter $\theta$ defined in \eqref{v4} satisfies
$q\xi\theta=1$, $N\xi\theta=2N/p$, and, since $q\in (q_1,q_\star)$
and $u_0$ satisfies \eqref{x1}, it follows from \eqref{v3} that
\beqn \|u(t)\|_1 \le C(u_0)\ (t-s)^{-2N/p}\ \|u(s)\|_1\,, \qquad
0\le s<t\,. \label{v8} \eeqn Let $B>0$ be a positive real number to
be determined later, $T>B$ and define
$$
m(T) := \sup_{t\in (0,T]}{\left\{ e^{t/B}\ \|u(t)\|_1 \right\}}\,.
$$
If $t\in (B,T]$, we infer from \eqref{v8} with $s=t-B\in (0,T]$ that
$$
e^{t/B}\ \|u(t)\|_1 \le C(u_0)\ B^{-2N/p}\ e^{t/B}\ \|u(t-B)\|_1 \le
C(u_0) e\ B^{-2N/p}\ m(T)\,,
$$
while, if $t\in (0,B]$, we have $e^{t/B}\ \|u(t)\|_1 \le e\
\|u_0\|_1$. Therefore, \bean
& & e^{t/B} \ \|u(t)\|_1 \le e\ \|u_0\|_1 + C(u_0)\ B^{-2N/p}\ m(T)\,, \qquad t\in (0,T]\,, \\
& & \left( 1 - \frac{C(u_0)}{B^{2N/p}} \right)\ m(T) \le e\
\|u_0\|_1\,. \eean Choosing $B$ suitably large such that
$B^{2N/p}\ge 2C(u_0)$ ensures that $m(T)$ is bounded from above by a
positive constant which does not depend on $T$. Consequently,
$\|u(t)\|_1\le C(u_0)\ e^{-t/B}$ for $t\ge 0$ which implies together
with \eqref{spirou} that $\|u(t)\|_\infty$ also decays at an
exponential rate with a possibly different constant.
\end{proof}

We now show that, at least for $p>p_c$, the exponential decay
obtained so far is optimal in the sense that the $L^1$-norm of $u$
cannot decay faster than exponentially. More precisely, we have the
following result:

\begin{proposition}\label{pr:v3}
If $p\in (p_c,2)$, $q=p/2$, and $u_0$ satisfies \eqref{assump}, then
there are positive constants $C_1(u_0)$ and $C_1'(u_0)$ depending on
$p$, $q$, $N$, and $u_0$ such that \beqn \|u(t)\|_1 +
\|u(t)\|_\infty \ge C_1'(u_0)\ e^{-C_1(u_0)t}\,, \quad t>0\,.
\label{lv1} \eeqn In addition, $\mathcal{P}=Q_\infty$.
\end{proposition}

\begin{proof}
Let $t>0$. By Proposition~\ref{le:u1}, we have
$$
\frac{d}{dt}\|u(t)\|_1 + \int |\nabla u(t,x)|^{p/2}\ dx = 0\,,
$$
while the gradient estimate \eqref{grad.est2} implies that
$$
|\nabla u(t,x)| = \frac{p}{2-p}\ u^{2/p}(t,x)\ \left| \nabla
u^{-(2-p)/p}(t,x) \right| \le C(u_0)\ u^{2/p}(t,x)\
\left( 1+ t^{-1/p} \right)\,.
$$
Combining the above two properties leads us to
$$
0 \le \frac{d}{dt}\|u(t)\|_1 + C(u_0)\ \left( 1+ t^{-1/p}
\right)\ \|u(t)\|_1\,,
$$
from which we readily conclude that $\|u(t)\|_1\ge \|u_0\|_1\
e^{-C(u_0)(t+t^{1/p})}$ for $t\ge 0$. On the one hand, this implies
that $\|u(t)\|_1 \ge \|u_0\|_1\ e^{-C(u_0)t}$ for $t\ge 1$, whence
\eqref{lv1}. On the other hand, we have $\|u(t)\|_1>0$ for all $t>0$
and we proceed as in the proof of Corollary~\ref{cor:positivity} to
show that $u(t,x)>0$ in $Q_\infty$.
\end{proof}

\medskip

\begin{proof}[Proof of Proposition~\ref{pr:pos}]
We check the first assertion which readily follows from
Proposition~\ref{nonextinction} and Corollary~\ref{cor:positivity}
when $p>p_c$ and $q>p/2$ and from Proposition~\ref{pr:v3} for
$p>p_c$ and $q=p/2$. Consider next the case $p=p_c$ and $q>p_c/2$. A
classical truncation argument ensures that there exists a
non-negative compactly supported function $\tilde{u}_0$ satisfying
\eqref{assump} and $\tilde{u}_0\le u_0$ in $\RR^N$. Denoting the
solution to \eqref{eq1} with initial condition $\tilde{u}_0$ by
$\tilde{u}$, we infer from the comparison principle that
$\tilde{u}\le u$ in $Q_\infty$. In addition, $\tilde{u}_0$ obviously
satisfies \eqref{u0b} for some $C_0>0$ and we are in a position to
apply Proposition~\ref{nonextinction} and
Corollary~\ref{cor:positivity} to $\tilde{u}$ and deduce that
$\|\tilde{u}(t)\|_1>0$ for all $t\ge 0$ and $\tilde{u}>0$ in
$Q_\infty$. Consequently, $u$ enjoys the same properties which
completes the proof of the first assertion in
Proposition~\ref{pr:pos}.

Next, the second assertion follows from Proposition~\ref{zerolimit}
if $q\in (0,q_\star]$ and from Proposition~\ref{positivity} if
$p>p_c$ ad $q>q_\star$. Finally, if $p=p_c$ and $q>q_\star$, there
is a non-negative compactly supported function $\tilde{u}_0$
satisfying \eqref{assump} and $\tilde{u}_0\le u_0$ in $\RR^N$. On
the one hand, the comparison principle guarantees that the solution
$\tilde{u}$ to \eqref{eq1} with initial condition $\tilde{u}_0$
satisfies $\tilde{u}\le u$ in $Q_\infty$. On the other hand,
$\tilde{u}_0$ clearly satisfies \eqref{u0b} for a suitable constant
$C_0$ and Proposition~\ref{positivity} ensures that
$\lim\limits_{t\to\infty}\|\tilde{u}(t)\|_1>0$. Combining these two
facts completes the proof of Proposition~\ref{pr:pos}.
\end{proof}

\subsection{Extinction}

To complete the proof of Theorem~\ref{th:idrext}, it remains to
establish that finite time extinction takes place when $p\ge p_c$
and $q\in (0,p/2)$. To this end, we need to handle separately and by
different methods the two cases: (a) $p\in (p_c,2)$ and $q\in
(q_1,p/2)$, (b) $p\in (p_c,2)$ and $q\in (0,q_1]$. Let us begin with
the case~(a) for which the proof uses Lemma~\ref{le:v1}.

\begin{proof}[Proof of Theorem~\ref{th:idrext}~(iii): $p\in (p_c,2)$ and $q\in (q_1,p/2)$]
In that case, we first observe that
$$
N\xi\theta > q \xi \theta = 1 + \frac{N\xi (p-2q)}{p-q}>1\,,
$$
the parameter $\theta$ being still defined in \eqref{v4}. Setting
$\lambda:=q/(N\xi\theta(p-q))$ and recalling that $q<p-q$ as $q<p/2$
and $u_0$ satisfies \eqref{x4} with $Q=q$, it follows from
\eqref{v3} that, for $s>0$, \bean
\tau(s) & := & \int_s^\infty \frac{\|u(t)\|_1^\lambda}{t}\ dt \le C(u_0)\ \|u(s)\|_1^{q\xi\theta\lambda}\  \int_s^\infty \frac{dt}{t (t-s)^{q/(p-q)}}\\
& \le & C(u_0)\ (-\tau'(s))^{q\xi\theta}\
s^{q\xi(N+1)(p-2q)/(p-q)}\,, \eean thus
$$\tau(s)^{1/(q\xi\theta)}\le-C(u_0)\ \tau'(s)\
s^{(p-2q)/(q_\star-q)}\,, $$ whence
$$
\tau'(s) + C(u_0)\ s^{-(p-2q)/(q_\star-q)}\
\tau(s)^{1/(q\xi\theta)}\le 0\,, \qquad s>0\,.
$$
Since
$$
\frac{p-2q}{q_\star-q} = 1 - \frac{1}{(N+1)\xi(q_\star-q)}<1\,,
$$
we infer from the above differential inequality that the function
$\tilde{\tau}: s \mapsto \tau\left( s^{(N+1)\xi(q_\star-q)} \right)$
satisfies
$$
\tilde{\tau}'(s) + C(u_0)\ \tilde{\tau}(s)^{1/(q\xi\theta)}\le 0\,,
\qquad s>0\,.
$$
Since $q\xi\theta>1$, we readily deduce from the above differential
inequality that $\tilde{\tau}(s)$ vanishes identically for $s$ large
enough and so do $\tau(s)$ and $\|u(s)\|_1$.
\end{proof}

We next turn to the remaining case for $p>p_c$ for which we cannot
use Lemma~\ref{le:v1}. We instead argue by comparison.

\begin{proof}[Proof of Theorem~\ref{th:idrext}~(iii): $p\in (p_c,2)$ and $q\in (0,q_1{]}$]
In that case, $q_1<p/2$ and, recalling that $Q\in (q_1,p/2)$ is
defined in \eqref{x4}, we put
$$
\tilde{u}(t,x) := \|\nabla u_0\|_\infty^{-(Q-q)/(Q-p+1)}\ u\left(
\|\nabla u_0\|_\infty^{((2-p)(Q-q))/(Q-p+1)}\ t , x \right)\,,
\qquad (t,x)\in Q_\infty\,.
$$
It follows from \eqref{eq1}, \eqref{a2}, and \eqref{prunelle} that
\bean
\partial_t\tilde{u}(t,x) & = & \|\nabla u_0\|_\infty^{-((p-1)(Q-q))/(Q-p+1)}\ \partial_t u\left( \|\nabla u_0\|_\infty^{((2-p)(Q-q))/(Q-p+1)}\ t , x \right) \\
& = & \|\nabla u_0\|_\infty^{-((p-1)(Q-q))/(Q-p+1)}\ (\Delta_p u - |\nabla u|^q)\left( \|\nabla u_0\|_\infty^{((2-p)(Q-q))/(Q-p+1)}\ t , x \right) \\
& \le & \Delta_p \tilde{u}(t,x) - \|\nabla u_0\|_\infty^{Q-q}\ |\nabla\tilde{u}(t,x)|^Q
\left\|\nabla u\left(\|\nabla u_0\|_\infty^{((2-p)(Q-q))/(Q-p+1)}\ t \right) \right\|_\infty^{q-Q} \\
& \le & \Delta_p \tilde{u}(t,x) - |\nabla\tilde{u}(t,x)|^Q\,, \eean
with $\tilde{u}(0)=U_0:=\|\nabla u_0\|_\infty^{-(Q-q)/(Q-p+1)}\
u_0$. Denoting the solution to \eqref{eq1}-\eqref{a2} with $Q$
instead of $q$ and $U_0$ instead of $u_0$ by $U$, the comparison
principle entails that $\tilde{u}\le U$ in $Q_\infty$. As $Q\in
(q_1,p/2)$ and $u_0$ satisfies \eqref{x4}, we already know that $U$
has the finite time extinction property by Theorem~\ref{th:idrext}.
Consequently, $\tilde{u}$ and also $u$ are identically zero after a
finite time.
\end{proof}

\noindent The other two extinction ranges, either
$p=p_c$ and $q\in(0,p_c/2)$, or $p\in(1,p_c)$ and $q>0$, have been
already considered in Theorem~\ref{th:warmup} and proved in Section~\ref{sec.decay}.

\subsection{A lower bound at the extinction time: $p\in (p_c,2)$ and $q\in (q_1,p/2)$}

It turns out that a simple modification of the proof of
Theorem~\ref{th:idrext} for $p\in (p_c,2)$ and $q\in (q_1,p/2)$
provides a lower bound on the $L^1$-norm and the $L^\infty$-norm of
$u(t)$ as $t$ approaches the extinction time $T_{\text{e}}$.

\begin{proposition}\label{pr:v2}
Assume that $p\in (p_c,2)$, $q\in (q_1,p/2)$, and that $u_0$
satisfies \eqref{assump} and \eqref{x4} (with $Q=q$). Denoting the
extinction time of the corresponding solution $u$ to
\eqref{eq1}-\eqref{a2} by $T_{\text{e}}$, we have \bear
C\ \left( T_{\text{e}} - t \right)^{(N+1)(q_\star-q)/(p-2q)} & \le & \|u(t)\|_1\,, \qquad t\in (0,T_{\text{e}})\,, \label{v9a}\\
C\ \left( T_{\text{e}} - t \right)^{(p-q)/(p-2q)} & \le &
\|u(t)\|_\infty\,, \qquad t\in (0,T_{\text{e}})\,. \label{v9b} \eear
\end{proposition}

\begin{proof}
By Theorem~\ref{th:idrext}~(iii), $T_{\text{e}}$ is finite and
$\|u(t)\|_1>0$ for $t\in \left[ 0,T_{\text{e}} \right)$. Setting
$\lambda=q/(N\xi\theta(p-q))$ with $\theta$ defined in \eqref{v4}
and recalling that $q<(p-q)$ as $q<p/2$, it follows from \eqref{v3}
that, for $s\in (0,T_{\text{e}})$, \bean
\tau(s) & := & \int_s^{T_{\text{e}}} \|u(t)\|_1^\lambda\ dt \le C(u_0)\ \|u(s)\|_1^{q\xi\theta\lambda}\  \int_s^{T_{\text{e}}}\frac{dt}{(t-s)^{q/(p-q)}}\\
& \le & C(u_0)\ (-\tau'(s))^{q\xi\theta}\ \left( T_{\text{e}} - s
\right)^{(p-2q)/(p-q)}\,, \eean from which we deduce the following
differential inequality:
\begin{equation*}
\tau(s)^{1/(q\xi\theta)}\le- C(u_0)\ \tau'(s)\ \left( T_{\text{e}} -
s \right)^{(p-2q)/(q(N+1)\xi(q_\star-q))}\,, \end{equation*} whence
$$
\tau'(s) + C(u_0)\ \left( T_{\text{e}} -
s\right)^{-(p-2q)/(q(N+1)\xi(q_\star-q))}\
\tau(s)^{1/(q\xi\theta)}\le 0\,, \qquad s\in \left( 0, T_{\text{e}}
\right)\,.
$$
Since
$$
\frac{1}{q\xi\theta} = 1 - \frac{N(p-2q)}{q(N+1)(q_\star-q)}<1
\quad\mbox{ and }\quad \frac{p-2q}{q(N+1)\xi(q_\star-q)} = 1 -
\frac{N\xi(p-2q) + q}{q(N+1)\xi(q_\star-q)}<1\,,
$$
the above differential inequality also reads
$$
\frac{d}{ds} \left[ \tau(s)^{N(p-2q)/(q(N+1)(q_\star-q))} - C(u_0)\
\left( T_{\text{e}} - s
\right)^{(N\xi(p-2q)+q)/(q(N+1)\xi(q_\star-q))} \right] \le 0
$$
for $s\in \left( 0,T_{\text{e}} \right)$. Integrating the above inequality with respect to $s$ over $\left( t,
T_{\text{e}} \right)$ for $t\in \left( 0, T_{\text{e}} \right)$
gives \bear
C(u_0)\ \left( T_{\text{e}} - t \right)^{(N\xi(p-2q)+q)/(q(N+1)\xi(q_\star-q))} & \le & \tau(t)^{N(p-2q)/(q(N+1)(q_\star-q))} \,, \nonumber\\
C(u_0)\ \left( T_{\text{e}} - t
\right)^{(N\xi(p-2q)+q)/(N\xi(p-2q))} & \le & \tau(t)\,. \label{v10}
\eear Owing to the time monotonicity \eqref{decaymass} of $\|u\|_1$,
we have \beqn \tau(t) = \int_t^{T_{\text{e}}} \|u(s)\|_1^\lambda\
ds\le \left( T_{\text{e}} - t \right)\ \|u(t)\|_1^\lambda\,, \qquad
t\in \left( 0, T_{\text{e}} \right)\,. \label{v11} \eeqn Combining
\eqref{v10} and \eqref{v11} gives \eqref{v9a}. Next, \eqref{v9b}
readily follows from \eqref{v2} and \eqref{v9a}.
\end{proof}

\section{Well-posedness}\label{se:wp}

In this section we study the existence and uniqueness of a solution
to \eqref{eq1}-\eqref{a2}. This is done through an
approximation process, in order to avoid the singularity in the
diffusion.

We begin by stating in a precise form the notion of a viscosity
solution to the singular equation \eqref{eq1}. The standard
definition has been adapted to deal with singular equations in
\cite{IS, OS}, by restricting the comparison functions. We follow
their approach. Let $\cf$ be the set of functions $f\in
C^{2}([0,\infty))$ satisfying
$$
f(0)=f'(0)=f''(0)=0, \ f''(r)>0 \ \hbox{for} \ \hbox{all} \ r>0,
\quad \lim\limits_{r\to0}|f'(r)|^{p-2}f''(r)=0.
$$
For example, $f(r)=r^{\sigma}$ with $\sigma>p/(p-1)>2$ belongs to
$\cf$. We introduce then the class $\ca$ of admissible comparison
functions $\psi\in C^2(Q_{\infty})$ defined as follows: $\psi\in\ca$
if, for any $(t_0,x_0)\in Q_{\infty}$ where $\nabla\psi(t_0,x_0)
=0$, there exist a constant $\delta>0$, a function $f\in\cf$, and a
modulus of continuity $\omega\in C([0,\infty))$, (that is, a
non-negative function satisfying $\omega(r)/r\to 0$ as $r\to 0$),
such that, for all $(t,x)\in Q_{\infty}$ with
$|x-x_0|+|t-t_0|<\delta$, we have
$$
|\psi(t,x)-\psi(t_0,x_0)-\partial_t\psi(t_0,x_0)(t-t_0)|\le f(|x-x_0|)+\omega(|t-t_0|).
$$
\begin{definition}\label{def:vs}
An upper semicontinuous function $u:Q_{\infty}\to\real$ is a
viscosity subsolution to \eqref{eq1} in $Q_{\infty}$ if, whenever
$\psi\in\ca$ and $(t_0,x_0)\in Q_{\infty}$ are such that
\begin{equation*}
u(t_0,x_0)=\psi(t_0,x_0), \quad u(t,x)<\psi(t,x), \ \mbox{for all}\
(t,x)\in Q_{\infty}\setminus\{(t_0,x_0)\},
\end{equation*}
then
\begin{equation}
\left\{\begin{array}{ll}\partial_t\psi(t_0,x_0)\leq\Delta_{p}\psi(t_0,x_0)-|\nabla\psi(t_0,x_0)|^{q} &
\ \hbox{if} \ \nabla\psi(t_0,x_0)\neq0,\\
\partial_t \psi(t_0,x_0)\leq 0 & \ \hbox{if} \
\nabla\psi(t_0,x_0)=0.\end{array}\right.
\end{equation}
A lower semicontinuous function $u:Q_{\infty}\to\real$ is a
viscosity supersolution to \eqref{eq1} in $Q_{\infty}$ if $-u$ is a
viscosity subsolution to \eqref{eq1} in $Q_{\infty}$. A continuous function
$u:Q_{\infty}\to\real$ is a viscosity solution to \eqref{eq1} in
$Q_{\infty}$ if it is a viscosity subsolution and supersolution.
\end{definition}

We refer to \cite{OS} for basic results about viscosity solutions;
in particular the comparison principle is \cite[Theorem 3.9]{OS} and
the stability property with respect to uniform limits is
\cite[Theorem 6.1]{OS}. We are now ready to state the main result of
this section.
\begin{theorem}\label{th:wp}
Given an initial condition $u_0$ satisfying \eqref{assump} there is
a unique non-negative viscosity solution $u$ to
\eqref{eq1}-\eqref{a2} which satisfies the gradient estimates stated
in Theorems~\ref{th:grad1}, \ref{th:grad2} and~\ref{th:grad3}
according to the range of $(p,q)$. In addition, $u$ is a weak
solution to \eqref{eq1}-\eqref{a2}, that is,
\begin{equation}\label{weakform}
\int (u(t,x)-u(s,x))\ \vartheta(x)\ dx + \int_s^t \int \left(
|\nabla u|^{p-2}\ \nabla u \cdot \nabla\vartheta + |\nabla u|^q\
\vartheta \right)\ dxd\tau = 0
\end{equation}
for $t>s\ge 0$ and all $\vartheta\in \mathcal{C}_0^\infty(\RR^N)$
and satisfies
\begin{equation}\label{decaymass}
\|u(t)\|_1 + \int_s^t \int |\nabla u(\tau,x)|^q\ dxd\tau \le
\|u(s)\|_1\,.
\end{equation}
\end{theorem}

\begin{remark}\label{rem:vp}
In fact the existence result can be extended to a larger class of
initial data, namely $u_0\in BC(\real^N)$. This can be proved by
further regularization and arguing as in \cite{GGK03}.
\end{remark}

The rest of the section is devoted to the proof of Theorem
\ref{th:wp}. This will be divided into several steps.

\subsection{Approximation}

In a first step, we have to introduce a regularization of
\eqref{eq1} in order to avoid the problems coming from the
singularity at points where $\nabla u=0$ and from the possible lack
of regularity of the solutions. For $\e\in(0,1/2)$, we let
\begin{equation}\label{approx}
a_{\e}(\xi):=(\xi+\e^2)^{(p-2)/2}, \quad
b_{\e}(\xi):=(\xi+\e^2)^{q/2}-\e^q, \qquad \xi\ge 0\,,
\end{equation}
and consider the following Cauchy problem
\begin{equation}\label{approx.eq}
\left\{\begin{array}{ll}\partial_{t}u_{\e}-\hbox{div}(a_{\e}(|\nabla
u_{\e}|^2)\nabla u_{\e})+b_{\e}(|\nabla u_{\e}|^2)=0, \ (t,x)\in Q_{\infty}\,,\\
u_{\e}(0,x)=u_{0\e}(x)+\e^{\gamma}, \
x\in\real^N\,,\end{array}\right.
\end{equation}
where $\gamma\in(0,p/4)\cap (0,q/2)$ is a small parameter such that
$\gamma<\min{\{ p-1, 1-k \}}$ and $u_{0\e}\in C^\infty(\RR^N)$ is a
non-negative smooth approximation of $u_0$ satisfying \beqn \|
u_{0\e}\|_\infty \le \|u_0\|_\infty \ \mbox{ and }\ \|\nabla
u_{0\e}\|_\infty \le (1+C(u_0) \e) \|\nabla u_0\|_\infty
\label{new0} \eeqn and such that $(u_{0\e})$ converges to $u_0$
uniformly in compact subsets of $\RR^N$. Further smallness
conditions on $\gamma$ and $\varepsilon$ will appear in the sequel
and will be stated wherever needed. By standard existence results
for quasilinear parabolic equations \cite{LSU}, \eqref{approx.eq}
has a unique classical solution $u_{\e}\in
C^{(3+\delta)/2,3+\delta}([0,\infty)\times\real^N)$ for some
$\delta\in(0,1)$. By comparison with constant solutions
$\e^{\gamma}$ and $\e^{\gamma}+\|u_0\|_{\infty}$, we find
\begin{equation}\label{approx.est}
\e^{\gamma}\le u_{\e}(t,x)\le\e^{\gamma}+\|u_0\|_{\infty}, \
(t,x)\in Q_{\infty}.
\end{equation}
We now turn to estimates for the gradient of $u_{\e}$. Let $\varphi$
be a $C^3$-smooth monotone function with inverse $\psi=\varphi^{-1}$
and set $\varrho=1/\psi'$. Defining $v_\e:=\varphi^{-1}(u_\e)$ and
$w_\e:=|\nabla v_\e|^2$, the regularity of $a_{\e}$, $b_{\e}$, and
$u_{\e}$ allows us to apply \cite[Lemma~2.1]{BtL08} and obtain that
$w_\e$ satisfies the differential inequality
\begin{equation}\label{new1}
\partial_t w_\e - A_\e w_\e - B_\e \cdot \nabla w_\e + 2 \tilde{R}_1^\e\ w_\e^2 + 2 \tilde{R}_2^\e\ w_\e \le 0 \;\mbox{ in }\; Q_\infty\,,
\end{equation}
with
\begin{eqnarray*}
A_\e w_\e & := & a_\e \Delta w_\e + 2 a_\e'\ (\nabla u_\e)^{t}D^{2}w_\e\nabla u_\e\,, \\
\tilde{R}_1^\e & := & - a_\e\ \left( \frac{\varphi''}{\varphi'} \right)' - \left( (N-1)\ \frac{(a_\e')^2}{a_\e} + 4 a_\e'' \right)\ (\varphi' \varphi'')^2\ w_\e^2 \\
& & \ - 2\ a_\e'\ \left( 2 (\varphi'')^2 + \varphi'\ \varphi''' \right) w_\e\,, \\
\tilde{R}_2^\e & := & \frac{\varphi''}{(\varphi')^2}\ \left( 2 b_\e' (\varphi')^2 w_\e - b_\e \right)\,,
\end{eqnarray*}
in which we have omitted to write the dependence of $a_\e$ and
$b_\e$ upon $|\nabla u_\e|^2$ and that of $\varphi$ upon $v_\e$.
Setting $g_\e:=(|\nabla u_{\e}|^2+\e^2)^{1/2}$, we have $|\nabla
u_{\e}|^2=g_\e^2-\e^2$ and we proceed as in Section~\ref{sec.grad}
to compute $\tilde{R}_1^\e$ and $\tilde{R}_2^\e$:
\begin{equation}\label{approx.R}
\begin{split}
&\tilde{R}_1^\e := (p-1)\ R_1^\e + \e^2\ R_{11}^\e \;\mbox{ with }\;  R_{1}^{\e}:=g_\e^{p-2}\left[k\ \varrho'(u_{\e})^2-(\varrho\varrho'')(u_{\e})\right],\\
&\tilde{R}_2^\e := (q-1)\ R_2^\e + R_{21}^\e \;\mbox{ with }\; R_2^{\e}:=\left( \frac{\varrho'}{\varrho} \right)(u_{\e})\ g_\e^q,
\end{split}
\end{equation}
and
\begin{equation}\label{r11}
\begin{split}
R_{11}^\e &:=\left[(p-2)\varrho\varrho''-(p-1)k(\varrho')^2\right](u_{\e})\ g_\e^{p-4}\\
&+\frac{(2-p)[2(N+7)-p(N+3)]}{4}\ \varrho'(u_{\e})^{2}\ g_\e^{p-6}\ (g_\e^2-\e^2) , \\
R_{21}^\e &:= \left( \frac{\varrho'}{\varrho} \right)(u_{\e})\ \left( \e^q - q \e^2 \ g_\e^{q-2} \right),
\end{split}
\end{equation}
After these preliminary calculations, we are ready to prove gradient
estimates for $u_{\e}$, that will give a rigorous proof of the
gradient estimates listed in Theorems~\ref{th:grad1},
\ref{th:grad2}, and~\ref{th:grad3} after passing to the limit $\e\to
0$ and a tool in the proof of well-posedness. Before the more
sophisticated estimates, let us notice that, taking
$\varrho(r)\equiv1$, we have
$R_1^{\e}=R_{11}^\e=R_{2}^{\e}=R_{21}^\e=0$ and the comparison
principle applied to \eqref{new1} and combined with \eqref{new0}
readily gives
\begin{equation}\label{approx.grad.est0}
\|\nabla u_{\e}(t)\|_{\infty}\leq\|\nabla
u_{0\e}\|_{\infty} \le (1+C(u_0) \e) \|\nabla u_0\|_\infty\,, \quad t\ge 0.
\end{equation}
Consequently,
\begin{equation}\label{new2}
\e \le g_\e \le \e + \|\nabla u_{0\e}\|_{\infty} \le \|\nabla u_0\|_\infty + C(u_0) \e \;\;\mbox{ in }\;\; Q_\infty.
\end{equation}

\subsection{Gradient estimates}\label{subsec.approx.grad}

In this subsection, we prove gradient estimates for $u_{\e}$. We
divide the proof into the same cases as in Section~\ref{sec.grad}.
In all cases, we will follow the four-step scheme: first estimate
the extra term $R_{11}^\e w_\e$, then the influence of the diffusion
term $R_1^{\e} w_\e^2$, then (if needed) the influence of the
absorption terms $R_2^{\e} w_\e$ and $R_{21}^{\e} w_\e$ and finally
find a suitable supersolution, as in the formal derivation performed
in Section~\ref{sec.grad}.

\subsubsection{$p>p_c$ and $q\ge p/2$.}\label{newsubsub1}

As in Section~\ref{subsec.grad.est1} we choose
\begin{equation*}
\varrho(z)=\left(\frac{p^2}{2(2k+p-2)}\right)^{1/p}\ z^{2/p},
\end{equation*}
and we obtain
\begin{equation*}
\begin{split}
R_{11}^\e &=\left(\frac{p^2}{2(2k+p-2)}\right)^{2/p}u_{\e}^{(4-2p)/p}\ g_\e^{p-4}\left\{-\frac{2(2-p)^2}{p^2}-\frac{4k(p-1)}{p^2}\right.\\& \quad\quad \left.+\frac{(2-p)[2(N+7)-p(N+3)]}{p^2}\ \frac{g_\e^2-\e^2}{g_\e^2}\right\}
\geq -C\ u_{\e}^{(4-2p)/p}\ g_\e^{p-4},
\end{split}
\end{equation*}
hence, since
\begin{equation}\label{est.w}
w_\e = \left| \nabla \varphi^{-1}(u_\e) \right|^2 = \frac{|\nabla u_\e|^2}{\varrho(u_\e)^2} = \frac{g_\e^2-\e^2}{\varrho(u_\e)^2}\,,
\end{equation}
\begin{equation*}
R_{11}^\e w_\e \geq -C\ u_{\e}^{(4-2p)/p}\ g_\e^{p-4}\ \frac{g_\e^2-\e^2}{\varrho(u_{\e})^2}=-C\ g_\e^{p-4}(g_\e^2-\e^2)\ u_{\e}^{-2}\geq-\frac{C}{\e^{2\gamma}}\ g_\e^{p-2}
\end{equation*}
by \eqref{approx.est}. Thus, from the formula \eqref{approx.R}, we
deduce
\begin{equation*}
\begin{split}
\tilde{R}_{1}^{\e}\ w_\e^2 &\geq
(p-1) \left( \frac{p^2}{2(2k+p-2)}\right)^{(2-p)/p}\ g_\e^{p-2}u_{\e}^{(4-2p)/p}w_\e^2-C_{1}\e^{2(1-\gamma)}g_\e^{p-2}w_\e\\
&\geq (p-1) \left( \frac{p^2}{2(2k+p-2)}\right)^{(2-p)/p}\
u_{\e}^{(4-2p)/p}\frac{g_\e^p-\e^2
g_\e^{p-2}}{\varrho(u_{\e})^2}w_\e-C_{1}\e^{p-2\gamma}w_\e \\
&\geq (p-1) \left( \frac{p^2}{2(2k+p-2)}\right)^{(2-p)/p}\ u_{\e}^{(4-2p)/p}\frac{g_\e^p-\e^p}{\varrho(u_{\e})^2}w_\e-C_{1}\e^{p-2\gamma}w_\e\\
&\geq (p-1) \left( \frac{p^2}{2(2k+p-2)}\right)^{(2-p)/p}\
u_{\e}^{(4-2p)/p}\varrho(u_{\e})^{p-2}w_\e^{1+p/2}-
C_{1}u_{\e}^{(4-2p)/p}\frac{\e^p}{\varrho(u_{\e})^2}w_\e\\&-C_{1}\e^{p-2\gamma}w_\e
\geq (p-1)\ w_\e^{1+p/2}-C_{1}\e^{p-2\gamma}w_\e,
\end{split}
\end{equation*}
where we have repeatedly used the lower bound in \eqref{new2},
\eqref{approx.est}, and \eqref{est.w}. We also have
\begin{equation}\label{R2est}
\tilde{R}_{2}^{\e}w_\e=Cu_{\e}^{-1}\left[\e^q-q\e^2 g_\e^{q-2}-(1-q)g_\e^q\right]w_\e.
\end{equation}
We need to treat in a different way the cases $q>1$ and $q<1$.

If $q>1$, we notice that $\tilde{R}_2^{\e}w_\e\geq 0$. Indeed, for
$q\ge 1$, we have $q\e^2 g_\e^{q-2} \le q \e g_\e^{q-1} \le \e^q +
(q-1) g_\e^q$ by Young's inequality. Hence, we can simply drop the
effect of this term and deduce from \eqref{new1} and the previous
lower bound on $\tilde{R}_1^{\e}$ that
\begin{equation*}
L_{\e}w_\e :=\partial_{t}w_\e-A_\e w_\e-B_\e\cdot\nabla w_\e+ 2(p-1)\
w_\e^{1+p/2}-C_{1}\e^{p-2\gamma}w_\e\leq0
\end{equation*}
in $Q_{\infty}$. It is then straightforward to check that the
function
\begin{equation*}
W_\e(t)=\left(\frac{2+pC_1\e^{p/2}}{2p(p-1)}\right)^{2/p}t^{-2/p}
\end{equation*}
is a supersolution for the differential operator $L_{\e}$ above in
$(0,\e^{(4\gamma-p)/2})\times\real^N$, provided we choose
$\gamma<p/4$. The comparison principle and the definition
\eqref{notations.cons} of $k$ then ensure that
\begin{equation}\label{approx.grad.est1}
\left|\nabla u_{\e}^{-(2-p)/p}(t,x)\right|\leq \left( \frac{2-p}{p}
\right)^{(p-1)/p}\ \eta^{1/p}\ \left( 1+ C_1\e^{p/2} \right)^{1/p}\
t^{-1/p}
\end{equation}
for any $(t,x)\in(0,\e^{(4\gamma-p)/2})\times\real^N$. Notice that
$4\gamma-p<0$ by the choice of $\gamma$, so that the time interval
of validity of \eqref{approx.grad.est1} increases to $(0,\infty)$ as
$\e\to 0$.

If $q\in [p/2,1)$, we can further estimate the right-hand side of
\eqref{R2est}, taking into account the lower bound $g_\e>\e$, which
implies
\begin{equation*}
R_{21}^\e \ge C\ u_\e^{-1}\ \left( \e^q-q\e^{2}g_\e^{q-2} \right) \ge (1-q) C\ u_\e^{-1}\ \e^q \geq 0,
\end{equation*}
while \eqref{approx.est} and \eqref{est.w} give
\begin{equation*}
\begin{split}
(q-1)\ R_2^{\e}w_\e&\geq
-C_3\ u_{\e}^{-1}\ g_\e^{q}\ w_\e\geq -C_3\ u_{\e}^{-1}\ \left( \e^2+\varrho(u_{\e})^{2}w_\e \right)^{q/2}\ w_\e\\
&\geq - C_3\ u_{\e}^{-1}\ \left( \e^{q}w_\e+\varrho(u_{\e})^qw_\e^{(2+q)/2} \right)  \\
&\geq - C_3\ \e^{q-\gamma}w_\e - C_3\ \left( \|u_0\|_{\infty}+\e^{\gamma} \right)^{(2q-p)/p}\ w_\e^{(2+q)/2},
\end{split}
\end{equation*}
where we have used the form of $\varrho$ and \eqref{approx.est}.
Combining this lower bound with the already obtained lower bound on
$\tilde{R}_{1}^{\e}$, we obtain
\begin{equation*}
\begin{split}
L_{\e}w_\e:=\partial_{t}w_\e -A_\e w_\e-B_\e\cdot\nabla w_\e+C_{1}w_\e^{(2+p)/2} & - C_{3}(\|u_0\|_{\infty}+\e^{\gamma})^{(2q-p)/p}w_\e^{(2+q)/2} \\
& -C_4\ \left( \e^{q-\gamma}+ \e^{p-2\gamma} \right)\ w_\e\leq 0
\end{split}
\end{equation*}
in $Q_\infty$.  We notice that the function
\begin{equation*}
W_\e(t)=C_5\
\left[\left(\|u_0\|_{\infty}+\e^{\gamma}\right)^{2(2q-p)/p(p-q)} +
\e^{2(p-2\gamma)/p} + \e^{2(q-\gamma)/p} \right]
+\left(\frac{4}{pC_1}\right)^{2/p}\ t^{-2/p}
\end{equation*}
is a supersolution for the differential operator $L_{\e}$ in
$Q_\infty$ for a sufficiently large constant $C_5$. By the
comparison principle, we obtain the following gradient estimate:
\begin{equation}\label{approx.grad.est2}
\left|\nabla u_{\e}^{-(2-p)/p}(t,x)\right|\leq
C\ \left[ \left(\|u_0\|_{\infty}+\e^{\gamma}\right)^{(2q-p)/p(p-q)} + \e^{(p-2\gamma)/p} + \e^{(q-\gamma)/p} + t^{-1/p} \right]
\end{equation}
for any $(t,x)\in Q_\infty$.

\subsubsection{$p>p_c$ and $q\in(0,p/2)$.}\label{newsubsub2}

As in Section~\ref{subsec.grad.est2}, we choose  the following function
\begin{equation*}
\varrho(z)=\left(\frac{p-q}{k+p-q-1}\right)^{1/(p-q)}z^{1/(p-q)},
\end{equation*}
recalling that $k+p-q-1>0$ in that case. We estimate
$\tilde{R}_{1}^{\e}$ and $\tilde{R}_{2}^{\e}$ in the same way as in
Section~\ref{newsubsub1}, the only significant difference stemming from the special form of $\varrho$. We have
\begin{equation*}
\begin{split}
R_{11}^\e &=\left(\frac{p-q}{k+p-q-1}\right)^{2/(p-q)}\frac{1}{(p-q)^2}g_\e^{p-4}u_{\e}^{2(q-p+1)/(p-q)}\left[-(2-p)(q-p+1)-k(p-1)\right.\\
&\quad\quad \left.+\frac{(2-p)(2(N+7)-p(N+3))}{4}
\frac{g_\e^2-\e^2}{g_\e^2}\right]\geq - C u_{\e}^{2(q-p+1)/(p-q)}g_\e^{p-4},
\end{split}
\end{equation*}
hence $R_{11}^\e w_\e \geq-C_1\e^{-2\gamma}g_\e^{p-2}$, a similar
estimate as in Section~\ref{newsubsub1} (and with exactly the same
proof relying on \eqref{approx.est} and \eqref{est.w}).
Consequently,  following the same steps as in
Section~\ref{newsubsub1},
\begin{equation*}
\begin{split}
\tilde{R}_{1}^{\e}\ w_\e^2
&\geq C_{2}g_\e^{p-2}u_{\e}^{2(q-p+1)/(p-q)}w_\e^2-C_{1}\e^{2(1-\gamma)}g_\e^{p-2}w_\e\\
&\geq C_{2}u_{\e}^{2(q-p+1)/(p-q)}\frac{g_\e^p-\e^2
g_\e^{p-2}}{\varrho(u_{\e})^2}w_\e-C_{1}\e^{p-2\gamma}w_\e\\&\geq C_{2}u_{\e}^{2(q-p+1)/(p-q)}\frac{g_\e^p-\e^p}{\varrho(u_{\e})^2}w_\e-C_{1}\e^{p-2\gamma}w_\e\\
&\geq
C_{2}u_{\e}^{2(q-p+1)/(p-q)}\varrho(u_{\e})^{p-2}w_\e^{(2+p)/2}-C_{2}u_{\e}^{2(q-p+1)/(p-q)}\frac{\e^p}{\varrho(u_{\e})^2}w_\e-C_{1}\e^{p-2\gamma}w_\e\\
&\geq C_{2}u_{\e}^{(2q-p)/(p-q)}w_\e^{(2+p)/2}-C_{1}\e^{p-2\gamma}w_\e.
\end{split}
\end{equation*}
We next estimate $\tilde{R}_2^{\e}$:
\begin{equation*}
\begin{split}
\tilde{R}_{2}^{\e}w_\e &=C_3 u_{\e}^{-1}\left[\e^q-q\e^2
g_\e^{q-2}-(1-q)g_\e^q\right]w_\e \geq -C_3 u_{\e}^{-1}g_\e^q w_\e\\&\geq - C_3 u_{\e}^{-1} (\e^2+\varrho(u_{\e})^{2}w_\e)^{q/2} w_\e \geq -C_{3} u_{\e}^{-1} \left( \e^{q}w_\e+\varrho(u_{\e})^q w_\e^{(2+q)/2} \right) \\
& \geq-C_{3} \left[ \e^{q-\gamma}w_\e +
u_{\e}^{(2q-p)/(p-q)}w_\e^{(2+q)/2} \right].
\end{split}
\end{equation*}
From \eqref{new1} and these estimates, and taking into account that
$\gamma<p/2<p-q$, we obtain that
\begin{equation*}
L_{\e}w_\e :=\partial_{t}w_\e-A_\e w_\e-B_\e\cdot\nabla
w_\e+C_2 u_{\e}^{(2q-p)/(p-q)} w_\e^{(2+q)/2}\left(w_\e^{(p-q)/2}-C_4\right)-C_5\e^{q-\gamma}w_\e\leq0.
\end{equation*}
We look for a supersolution for $L_{\e}$ of the form
$W_\e(t)=\lambda +\mu t^{-2/p}$. Proceeding as in
Sections~\ref{subsec.grad.est2} and~\ref{newsubsub1}, we find that
\begin{equation*}
\begin{split}
W_\e(t) & =\left( \frac{4C_5}{C_2} \right)^{2/p}\ \left( \|u_0\|_\infty+ \e^\gamma \right)^{2(p-2q)/p(p-q)}\ \e^{2(q-\gamma)/p} + (4C_4)^{2/(p-q)} \\
& + \left(\frac{4}{pC_2}\right)^{2/p}\ \left(\|u_0\|_{\infty}+\e^{\gamma}\right)^{2(p-2q)/p(p-q)}\ t^{-2/p}
\end{split}
\end{equation*}
is a supersolution in $Q_\infty$. We thus obtain the following
gradient estimate
\begin{equation}\label{approx.grad.est3}
|\nabla u_{\e}(t,x)|u_{\e}(t,x)^{-1/(p-q)}\leq
C\left[1+\left(\|u_0\|_{\infty}+\e^{\gamma}\right)^{(p-2q)/p(p-q)}
\left( \e^{(q-\gamma)/p} + t^{-1/p} \right)\right],
\end{equation}
for any $(t,x)\in Q_\infty$. This is the approximation of
\eqref{part.grad.est2}, and the discussion with respect to the sign
of $p-1-q$ is the same as in Section~\ref{subsec.grad.est2} and is
omitted here.

\subsubsection{$p=p_c$.}\label{newsubsub3}

We follow the same general strategy as in the previous cases. The
computations are slightly different since logarithmic terms appear
in the choice of $\varrho$.

For $q>p_c/2$, we take
\begin{equation*}
\varrho(z)=z^{(N+1)/N}(\log M_{\e}-\log z)^{(N+1)/2N}, \quad
M_{\e}=e(\|u_0\|_{\infty}+\e^{\gamma}).
\end{equation*}
Let us notice first that, by \eqref{approx.est},
\begin{equation}\label{ineq}
1\leq\log M_{\e}-\log u_{\e} .
\end{equation}
On the one hand, owing to \eqref{ineq},
\begin{equation*}
\begin{split}
R_1^\e &= \frac{N+1}{4N}\ u_\e^{2/N}\ \left[ 2 (\log M_\e - \log u_\e)^{1/N} + (\log M_\e - \log u_\e)^{-(N-1)/N} \right] \\
& \ge \frac{N+1}{2N}\ u_\e^{2/N}\ g_\e^{p_c-2}   (\log M_\e - \log u_\e)^{1/N} .
\end{split}
\end{equation*}
On the other hand, after direct, but rather long computations, and
dropping, as usual, the last term in the expression \eqref{r11} of
$R_{11}^\e$, we deduce from \eqref{ineq} that
\begin{equation*}
\begin{split}
R_{11}^\e &\geq \frac{N+1}{2N^2}u_{\e}^{2/N}g_\e^{p_c-4}
\left[-2(\log M_{\e}-\log u_{\e})^{(N+1)/N}+ \frac{4N+2}{N+1}\ (\log
M_{\e}-\log u_{\e})^{1/N}\right.\\& \quad\quad \left.+\frac{N-1}{2(N+1)}(\log
M_{\e}-\log
u_{\e})^{-(N-1)/N}\right]\\
&\geq - \frac{N+1}{N^2}\ u_{\e}^{2/N}\ g_\e^{p_c-4} \left( \log M_\e
- \log u_\e \right)^{(N+1)/N},
\end{split}
\end{equation*}
Consequently, thanks to \eqref{approx.est} and \eqref{est.w}, we
have
\begin{equation*}
\begin{split}
R_{11}^\e w_\e &\geq - \frac{N+1}{N^2}\ u_{\e}^{2/N}\ g_\e^{p_c-4}\ \left( \log M_\e - \log u_\e \right)^{(N+1)/N} \ \frac{g_\e^2-\e^2}{\varrho(u_\e)^2} \\
& \geq -C_1 u_{\e}^{-2} g_\e^{p_c-4} (g_\e^2-\e^2) \geq -C_{1}\ \e^{-2\gamma} g_\e^{p_c-2}.
\end{split}
\end{equation*}
Using again \eqref{approx.est}, \eqref{new2}, \eqref{est.w}, and
\eqref{ineq}, we can now estimate
\begin{equation*}
\begin{split}
\tilde{R}_{1}^{\e} w_\e^2&\geq
C_2 g_\e^{p_c-2} u_{\e}^{2/N}\ \left( \log M_\e - \log u_\e \right)^{1/N} \ w_\e^2 - C_{1}\e^{2(1-\gamma)} g_\e^{p_c-2} w_\e\\
&\geq
C_2 u_{\e}^{2/N} \frac{g_\e^{p_c}-\e^2 g_\e^{p_c-2}}{\varrho(u_{\e})^2}\ \left( \log M_\e - \log u_\e \right)^{1/N} \ w_\e-C_{1}\e^{p_c-2\gamma}w_\e\\
&\geq
C_2u_{\e}^{2/N}\ \left( \log M_\e - \log u_\e \right)^{1/N}\ \left[ \varrho(u_{\e})^{p_c-2}\ w_\e^{(2+p_c)/2}-\frac{\e^{p_c}}{\varrho(u_{\e})^2}w_\e \right]
-C_1\e^{p_c-2\gamma}w_\e\\
&\geq C_2 w_\e^{(2+p_c)/2} - C_2 u_{\e}^{-2}\ (\log M-\log
u_{\e})^{-1} \e^{p_c}\ w_\e - C_1 \e^{p_c-2\gamma}\ w_\e\\
&\geq C_{2}w_\e^{(2+p_c)/2}-C_1\e^{p_c-2\gamma}w_\e.
\end{split}
\end{equation*}
It remains to estimate $\tilde{R}_2^{\e}$. By direct computation, we find
\begin{equation}\label{new3}
\tilde{R}_2^{\e}=\frac{N+1}{2N u_\e}\ \frac{2\log M_{\e}-2\log u_{\e}-1}{\log
M_{\e}-\log u_{\e}}\ \left[\e^q-q\e^2g_\e^{q-2}-(1-q)g_\e^q\right].
\end{equation}
If $q\geq1$, we have $\e^q - q \e^2 g_\e^{q-2}\ge 0$ (as in
Section~\ref{newsubsub1}), which, together with \eqref{ineq},
implies $\tilde{R}_2^{\e}\geq0$. We can simply drop this term and
end up with
\begin{equation*}
L_{\e}w_\e :=\partial_{t}w_\e-A_\e w_\e-B_\e\cdot\nabla w_\e + C_{2} w_\e^{(2+p_c)/2}-C_1\e^{p_c-2\gamma}w_\e\leq 0
\end{equation*}
in $Q_{\infty}$ by \eqref{new1}. We then argue as in
Section~\ref{newsubsub1} to check that, thanks to the choice of
$\gamma$, the function
\begin{equation*}
W_\e(t)=\left(\frac{2+p_c C_1 \e^{p_c/2}}{p_c
C_2}\right)^{2/p_c}t^{-2/p_c}
\end{equation*}
is a supersolution for the differential operator $L_{\e}$ in
$(0,\e^{(4\gamma-p_c)/2})\times\real^N$. The comparison principle then ensures that
\begin{equation}\label{approx.grad.est4}
\left|\nabla u_{\e}^{-1/N}(t,x)\right|\leq
C\left(1+\e^{N/(N+1)}\right)^{1/p_c}(\log M_{\e}-\log u_{\e}(t,x))^{1/p_c}t^{-1/p_c}
\end{equation}
for any $(t,x)\in (0,\e^{(4\gamma-p_c)/2})\times\real^N$.

If $q\in(p_c/2,1)$, we have to estimate $\tilde{R}_2^{\e}$ more
precisely. Since the mapping $z\mapsto (2z-1)/z$ is increasing in
$(0,\infty)$ and $\e^q - q \e^2 g_\e^{q-2} \ge (1-q) \e^q \ge 0$, it
follows from \eqref{approx.est}, \eqref{ineq}, and \eqref{new3} that
\begin{equation*}
\begin{split}
\tilde{R}_2^{\e} w_\e &\geq - \frac{(1-q)(N+1)}{2Nu_{\e}}\ \frac{2\log M_\e - 2 \log u_\e - 1 }{\log M_\e - \log u_\e}\ g_\e^q\ w_\e \geq - C_3 u_{\e}^{-1} g_\e^q w_\e\\
&\geq -C_3 u_{\e}^{-1} (\e^2+\varrho(u_{\e})^{2}w_\e )^{q/2} w_\e \geq -C_3 u_{\e}^{-1}(\e^{q}w_\e+\varrho(u_{\e})^q w_\e^{(2+q)/2})\\
&\geq -C_3 \e^{q-\gamma} w_\e - C_3 u_{\e}^{(q(N+1)-N)/N} (\log M_{\e}-\log u_{\e})^{q(N+1)/2N} w_\e^{(2+q)/2}.
\end{split}
\end{equation*}
We go on as in Section~\ref{subsec.grad.est3} by noticing that the
function
\begin{equation*}
z\mapsto z^{(q(N+1)-N)/N}(\log M_{\e}-\log z)^{q(N+1)/2N}
\end{equation*}
attains its maximum in the interval
$(0,\|u_{0}\|_{\infty}+\e^{\gamma})$ at $(\|u_0\|_\infty+\e^\gamma)
e^{-(N\xi-1)/2}$, hence we can write:
\begin{equation*}
\tilde{R}_2^{\e}\geq-C_4\e^{q-\gamma} w_\e - C_4\ (\|u_0\|_{\infty}+\e^{\gamma})^{(q(N+1)-N)/N}\ w_\e^{(2+q)/2}.
\end{equation*}
It follows that
\begin{equation*}
L_{\e}w_\e :=\partial_{t}w_\e-A_\e w_\e-B_\e\cdot\nabla
w_\e+C_2w_\e^{(2+p_c)/2}-C_4(\|u_0\|_{\infty}+\e^{\gamma})^{1/N\xi}w_\e^{1+q/2}-C_4\e^{q-\gamma}w_\e\leq0
\end{equation*}
in $Q_{\infty}$ since $\gamma<p_c-1<p_c-q$. We notice that the
function
\begin{equation*}
\begin{split}
W_\e(t) &=\left( \frac{4C_4}{C_2} \right)^{2/(p_c-q)}\ \left(\|u_0\|_{\infty}+\e^{\gamma}\right)^{2/(p_c-q)N\xi} + \left( \frac{4C_4}{C_2} \right)^{2/p_c}\ \e^{2(q-\gamma)/p_c}\\
& + \left(\frac{4}{p_c C_2}\right)^{2/p_c}t^{-2/p_c}
\end{split}
\end{equation*}
is a supersolution in $Q_\infty$. By the comparison principle, we obtain
\begin{equation*}
\left|\nabla u_{\e}^{-1/N}(t,x)\right| \leq
C\ \left[ \left( \|u_0\|_{\infty}+\e^{\gamma}\right)^{1/(p_c-q)N\xi} + \e^{(q-\gamma)/p_c} + t^{-1/p_c} \right]\
\left( \log\left( \frac{M_\e}{u_\e(t,x)} \right) \right)^{1/p_c}
\end{equation*}
for any $(t,x)\in Q_\infty$.

\medskip

For $q=p_c/2$, following the idea in Section~\ref{subsec.grad.est3}, we choose
\begin{equation*}
\varrho(z)=z^{(N+1)/N}(\log M_{\e}-\log z)^{(N+1)/N}, \quad
M_{\e}=e(\|u_0\|_{\infty}+\e^{\gamma}).
\end{equation*}
Proceeding as in the previous cases, we infer from
\eqref{approx.est}, \eqref{est.w}, and \eqref{ineq} that
\begin{eqnarray*}
R_{11}^\e w_\e & \ge & - \frac{2(N+1)}{N^2}\ \e^{-2\gamma}\ g_\e^{p_c-2} \,, \\
R_1^\e & \ge & \frac{N+1}{N}\ u_\e^{2/N}\ \left( \log M_\e - \log u_\e \right)^{(N+2)/N}\ g_\e^{p_c-2}\,,
\end{eqnarray*}
so that
\begin{equation*}
\tilde{R}_1^\e\ w_\e^2 \ge \frac{N-1}{N}\ \left( \log M_\e - \log u_\e \right)\ w_\e^{(2+p_c)/2} -C_1\ \e^{p_c-2\gamma}\ w_\e\,,
\end{equation*}
while
\begin{equation*}
\tilde{R}_2^\e\ w_\e \ge - \frac{\e^{q-\gamma}}{N}\ w_\e - \frac{1}{N}\ \left( \log M_\e - \log u_\e \right)\ w_\e^{(2+q)/2} \,.
\end{equation*}
Using a comparison argument as before we end up with the following
estimate
\begin{equation*}
\left|\nabla u_\e^{-1/N}(t,x)\right|\leq C\ \left( \log M_{\e}-\log u_{\e}(t,x) \right)^{(N+1)/N}\ \left( 1+\e^{(q-\gamma)/p_c} + t^{-1/p_c} \right)
\end{equation*}
for any $(t,x)\in Q_\infty$.

\medskip

Finally, if $q\in (0,p_c/2)$, we proceed as in
Section~\ref{newsubsub2} to show that \eqref{approx.grad.est3} holds
true.

\subsubsection{$p<p_c$ and $q>1-k$.}\label{newsubsub4}

We slightly modify the function $\varrho$ from the formal proof in
Section~\ref{subsec.grad.est4} and define the function $\varrho_\e$
by
\begin{equation}\label{new11}
\left( \frac{(2-p-2k) K_\e^p}{2} \right)^{1/2}\ \int_0^{\varrho_\e(r)/K_\e} \frac{dz}{z^k \left( 1- z^{2-p-2k} \right)^{1/2}} = r
\end{equation}
for $r\in \left[ 0, \|u_0\|_\infty +\e^\gamma \right]$, where
\begin{equation}\label{new12}
\left( \frac{(2-p-2k) K_\e^p}{2} \right)^{1/2}\ \int_0^1 \frac{dz}{z^k \left( 1- z^{2-p-2k} \right)^{1/2}} = \|u_0\|_\infty + \e^\gamma\,.
\end{equation}
Observe that $K_\e = \kappa \left( \|u_0\|_\infty +\e^\gamma
\right)^{2/p}$ and $K_\e\to K_0$ as $\e\to 0$, the constants
$\kappa$ and $K_0$ being defined in \eqref{est.C}. It readily
follows from \eqref{new11} that $\varrho_\e$ solves \eqref{ODE2}
with $K_\e$ instead of $K_0$ and thus \eqref{ODE} and
\begin{equation}\label{new13}
\varrho_\e(r) \le C\ K_\e^{(2-p-2k)/2(1-k)}\ r^{1/(1-k)}\,, \quad r\in \left[ 0, \|u_0\|_\infty +\e^\gamma \right]\,.
\end{equation}

Now, omitting as before the last term in $R_{11}^\e$ since it is
non-negative, we deduce from \eqref{ODE} and \eqref{ODE2} that
\begin{equation*}
\begin{split}
R_{11}^\e &\geq \left[ (2-p)\ \left(k \varrho_\e'(u_{\e})^2-\varrho_\e''(u_{\e})\varrho_\e(u_{\e}) \right) -k\ \varrho_\e'(u_{\e})^2 \right]\ g_\e^{p-4}\\
& \geq\left[ (2-p)\ \varrho_\e(u_{\e})^{2-p} - C\ K_\e^{2-p-2k}\ \varrho_\e(u_{\e})^{2k} \right]\ g_\e^{p-4} \\
& \geq - C\ K_\e^{2-p-2k}\  \varrho_\e(u_{\e})^{2k}\ g_\e^{p-4}.
\end{split}
\end{equation*}
We then infer from \eqref{new2}, \eqref{est.w}, \eqref{new13}, and
the positivity of $2-p-2k>0$ that
\begin{equation}\label{new14}
R_{11}^\e\  w_\e \geq - C\ K_\e^{2-p-2k}\ \varrho_\e(u_{\e})^{2k-2}\ g_\e^{p-4}\ \left( g_\e^2 - \e^2 \right)
\geq - C\ K_\e^{2-p-2k}\ \varrho_\e(u_{\e})^{2k-2}\ g_\e^{p-2}\,.
\end{equation}
Since $k<1$ and $\varrho_\e$ is increasing, we deduce from
\eqref{approx.est} that
\begin{equation}\label{new15}
\varrho_\e(u_{\e})^{2k-2} \le \varrho_\e\left( \e^\gamma \right)^{2k-2}\,.
\end{equation}
Now, on the one hand, as $2-p-2k>0$, we deduce from \eqref{new11} that
\begin{eqnarray*}
\e^\gamma & = & \left( \frac{(2-p-2k) K_\e^p}{2} \right)^{1/2}\ \int_0^{\varrho_\e(\e^\gamma)/K_\e} \frac{dz}{z^k \left( 1- z^{2-p-2k} \right)^{1/2}} \\
& \le & \left( \frac{2-p-2k}{2} \right)^{1/2}\ K_\e^{1-k}\ \int_0^{\varrho_\e(\e^\gamma)/K_\e} \frac{dz}{z^k \left( K_\e^{2-p-2k} - \varrho_\e(\e^\gamma)^{2-p-2k} \right)^{1/2}} \\
& \le & \left( \frac{2-p-2k}{2 (1-k)^2} \right)^{1/2}\ \frac{\varrho_\e(\e^\gamma)^{1-k}}{\left( K_\e^{2-p-2k} - \varrho_\e(\e^\gamma)^{2-p-2k} \right)^{1/2}}
\end{eqnarray*}
On the other hand, using again the positivity of $2-p-2k$ and $1-k$
and \eqref{new13}, we find that
\begin{equation}\label{new16}
\varrho_\e(\e^\gamma) \le C\ K_\e^{(2-p-2k)/2(1-k)}\ \e^{\gamma/(1-k)} \le \frac{1}{2^{1/(2-p-2k)}}\ K_\e\,,
\end{equation}
provided $\e\le \e_0(\|u_0\|_\infty)$ is chosen suitably small.
Combining \eqref{new15} and \eqref{new16} yields
$$
\e^\gamma \le \left( \frac{2-p-2k}{(1-k)^2} \right)^{1/2}\ \varrho_\e(\e^\gamma)^{1-k} \  K_\e^{-(2-p-2k)/2}\,.
$$
Consequently,
\begin{equation}\label{new17}
\varrho_\e(\e^\gamma)\ge C\ \e^{\gamma/(1-k)}\ K_\e^{(2-p-2k)/2(1-k)}
\end{equation}
which, together with \eqref{new2}, \eqref{new14}, and \eqref{new15} gives
$$
R_{11}^\e\  w_\e \geq - C\ K_\e^{2-p-2k}\ K_\e^{-(2-p-2k)}\ \e^{-2\gamma}\ g_\e^{p-2} \geq -C\ \e^{p-2-2\gamma} \,.
$$
Turning to $R_{1}^{\e}$, it follows from \eqref{ODE},
\eqref{approx.est}, \eqref{new2}, \eqref{est.w}, the monotonicity of
$\varrho_\e$, and \eqref{new17} that
\begin{equation*}
\begin{split}
R_1^{\e}\ w_\e^2 & = \varrho_\e(u_\e)^{2-p}\ g_\e^{p-2}\ w_\e^2 = \varrho_\e(u_\e)^{-p}\ g_\e^{p-2}\ \left( g_\e^2 - \e^2 \right)\ w_\e \\
& \geq \varrho_\e(u_\e)^{-p}\ g_\e^p\ w_\e - \e^p\ \varrho_\e(u_\e)^{-p}\ \ w_\e \geq w_\e^{(2+p)/2} - \e^p\ \varrho_\e(\e^\gamma)^{-p}\ \ w_\e \\
& \geq w_\e^{(2+p)/2} - C\ \e^{p(1-k-\gamma)/(1-k)}\ K_\e^{-p(2-p-2k)/(2-2k)}\ \ w_\e\,.
\end{split}
\end{equation*}
Gathering the above lower bounds on $R_1^\e$ and $R_{11}^\e$, we are lead to
\begin{equation*}
\begin{split}
\tilde{R}_1^{\e}\ w_\e^2& \geq 2(p-1)\ w_\e^{(2+p)/2} - C \left[ \e^{p(1-k-\gamma)/(1-k)}\ K_\e^{-p(2-p-2k)/(2-2k)} + \e^{p-2\gamma} \right]\ w_\e.
\end{split}
\end{equation*}

For $q\geq1$, the influence of $\tilde{R}_2^{\e}$ is a positive term
thanks to the monotonicity of $\varrho_\e$ (as in the previous
cases) and can be omitted. We obtain that
\begin{equation*}
L_{\e} w_\e:=\partial_{t} w_\e- A_\e w_\e - B_\e \cdot\nabla w_\e +
2(p-1)\ w_\e^{(2+p)/2} -C_1\ \mu_\e^2\ w_\e \leq 0 \;\;\mbox{ in
}\;\; Q_\infty\,,
\end{equation*}
with
\begin{equation}\label{new20}
\mu_\e := \e^{p(1-k-\gamma)/(1-k)}\ K_\e^{-p(2-p-2k)/(2-2k)} + \e^{p-2\gamma} \mathop{\longrightarrow}_{\e\to 0} 0
\end{equation}
thanks to the choice of $\gamma$. By noticing that
\begin{equation*}
W_\e(t)=\left(\frac{1+p C_1 \mu_\e}{2p(p-1)}\right)^{2/p}t^{-2/p}\,,
\qquad t>0\,,
\end{equation*}
is a supersolution for the differential operator $L_{\e}$ in
$(0,\mu_\e^{-1})\times\real^N$. The comparison principle then
implies that
\begin{equation*}
\left|\nabla u_{\e}(t,x)\right| \leq C\ \varrho_\e(u_{\e}(t,x))\ (1+\mu_\e)^{1/p}\ t^{-1/p}, \quad (t,x)\in (0,\mu_\e^{-1})\times\real^N)\,,
\end{equation*}
whence
\begin{equation}\label{approx.grad.est7}
\left|\nabla u_{\e}(t,x)\right| \leq C\ \left( \|u_0\|_\infty + \e^\gamma \right)^{(2-p-2k)/p(1-k)}\ u_\e(t,x)^{1/(1-k)}\ (1+\mu_\e)^{1/p}\ t^{-1/p}
\end{equation}
for any $(t,x)\in(0,\mu_\e^{-1})\times\real^N$. This is the
approximation giving, in the limit, the estimates in
Section~\ref{subsec.grad.est4}.

For $q\in[1-k,1)$, we necessarily have $p>p_{sc}=2(N+1)/(N+3)$ and,
recalling that $k<1$, it follows from \eqref{ODE2},
\eqref{approx.est}, \eqref{new2}, \eqref{new17}, and the
monotonicity of $\varrho_\e$ that
\begin{equation*}
\begin{split}
\tilde{R}_2^{\e}\ w_\e & \geq - (1-q)\ \frac{\varrho_\e'(u_{\e})}{\varrho_\e(u_{\e})}\ g_\e^q\ w_\e \geq - C\ K_\e^{(2-p-2k)/2}\ \varrho_\e(u_\e)^{k-1}\ g_\e^q\ w_\e \\
& \geq - C\ K_\e^{(2-p-2k)/2}\ \varrho_\e(u_\e)^{k-1}\ \left( \e^q + \varrho_\e(u_\e)^q\ w_\e^{q/2} \right)\ w_\e \\
& \geq - C\ K_\e^{(2-p-2k)/2}\ \left[ \varrho_\e(\e^\gamma)^{k-1}\ \e^q\ w_\e + \varrho_\e(u_\e)^{q+k-1}\ w_\e^{(2+q)/2} \right]\\
& \geq -C\ \e^{q-\gamma}\ w_\e - C\ K_\e^{(2-p-2k)/2}\ \varrho_\e(\|u_0\|_\infty +\e^\gamma)^{q+k-1}\ w_\e^{(2+q)/2} \\
&\geq - C_2\ \left[ \e^{q-\gamma}\ w_\e + K_\e^{(2q-p)/2}\ w_\e^{(2+q)/2} \right]\,.
\end{split}
\end{equation*}
Combining this lower bound with that for $\tilde{R}_1^\e\ w_\e^2$
established above, we realize that
\begin{equation*}
\begin{split}
L_{\e} w_\e:=\partial_{t} w_\e- A_\e w_\e - B_\e \cdot\nabla w_\e &+ 2(p-1)\ w_\e^{(2+p)/2} \\
& - C_2\ K_\e^{(2q-p)/2}\ w_\e^{(2+q)/2} - \left(  C_1\ \mu_\e^2 + C_2\ \e^{q-\gamma} \right)\ w_\e \leq 0
\end{split}
\end{equation*}
in $Q_\infty$ with $\mu_\e$ defined by \eqref{new20}. We next observe that the function
\begin{equation*}
W_\e(t) = \left(  C_1\ \mu_\e^2 + C_2\ \e^{q-\gamma} \right)^{2/p} + \left( \frac{C_2 K_\e^{(2q-p)/2}}{p-1} \right)^{2/(p-q)} + \left( \frac{2}{p(p-1)} \right)^{2/p}\ t^{-2/p}
\end{equation*}
is a supersolution for the differential operator $L_\e$ in $Q_\infty$ and deduce from the comparison principle and \eqref{new13} that
\begin{equation}\label{approx.grad.est8}
\begin{split}
& |\nabla u_{\e}(t,x)|\ u_\e(t,x)^{-1/(1-k)}\ \left( \|u_0\|_\infty + \e^\gamma \right)^{-(2-p-2k)/p(1-k)} \\
& \leq C\ \left( \mu_\e^{2/p} + \e^{(q-\gamma)/p} + \left( \|u_0\|_\infty + \e^\gamma \right)^{(2q-p)/p(p-q)} +t^{-1/p} \right)
\end{split}
\end{equation}
for any $(t,x)\in Q_\infty$.

\subsection{A gradient estimate related to the Hamilton-Jacobi term}

We prove, using the same approximation as before, the gradient
estimates \eqref{grad.estHJ} and \eqref{grad.estHJ2} formally
established in Section~\ref{subsec.grad.HJ}. As already mentioned,
we assume for simplicity $p>p_{sc}=2(N+1)/(N+3)$ and divide the
proof into two cases.

\subsubsection{$q\in (0,1)$.}\label{newsubsub5}

We set $\varrho(z)=-2(M_{\e}-z)^{1/2}$ for $z\in [0,M_\e]$, where
$M_{\e}:=\|u_0\|_{\infty}+2\e^{\gamma}$. On the one hand, we have
\begin{eqnarray*}
R_{1}^{\e} & = & (k+1)\ g_\e^{p-2}\ (M_{\e}-u_{\e})^{-1}\ge 0\,, \\
R_{11}^\e & \ge & \frac{(N+3)(2-p)^2}{4}\ g_\e^{p-4}\ (M_{\e}-u_{\e})^{-1}\ge 0\,.
\end{eqnarray*}
On the other hand, by \eqref{approx.est}, we have
\begin{equation*}
\begin{split}
R_2^{\e}w_\e &=\frac{1}{2}(M_{\e}-u_{\e})^{-1}\left[(1-q)g_\e^q+q\e^2g_\e^{q-2}-\e^q\right]w_\e\\
&\geq\frac{1-q}{2}(M_{\e}-u_{\e})^{-1}\left(\e^2+\varrho(u_{\e})^2w_\e\right)^{q/2}w_\e - \e^q\ (M_{\e}-u_{\e})^{-1}\ w_\e\\
&\geq \frac{1-q}{2}\ (M_{\e}-u_{\e})^{-1}\ |\varrho(u_{\e})|^q\ w_\e^{(2+q)/2} - \e^{q-\gamma}\ w_\e\\
&\geq 2^{q-1} (1-q)\ (M_{\e}-u_{\e})^{(q-2)/2}\ w_\e^{(2+q)/2} - \e^{q-\gamma}\ w_\e \\
&\geq 2^{q-1} (1-q)\ (\e^{\gamma}+\|u_0\|_{\infty})^{(q-2)/2}\
w_\e^{(2+q)/2} - \e^{q-\gamma}\ w_\e,
\end{split}
\end{equation*}
where we have used the bounds $\e^{\gamma}\leq M_{\e}-u_{\e}\leq
\e^{\gamma}+\|u_0\|_{\infty}$. Therefore
\begin{equation*}
L_{\e}w_\e :=\partial_{t}w_\e - A_\e w_\e - B_\e \cdot \nabla w_\e + C_1\ (\e^{\gamma}+\|u_0\|_{\infty})^{(q-2)/2}\ w_\e^{(2+q)/2} - \e^{q-\gamma}\ w_\e \leq 0
\end{equation*}
in $Q_\infty$. Now, the following function
\begin{equation*}
W_\e(t):=\left(\e^{\gamma}+\|u_0\|_{\infty}\right)^{(2-q)/q}\left(\frac{2+q\e^{q/2}}{qC_1}\right)^{2/q}t^{-2/q}\,,
\quad t>0\,,
\end{equation*}
is a supersolution for the differential operator $L_{\e}$ in
$(0,\e^{(2\gamma-q)/2})\times\real^N$. We then deduce from the
comparison principle that
\begin{eqnarray*}
\left|\nabla u_{\e}(t,x)\right| & \leq & C\ (M_{\e}-u_{\e}(t,x))^{1/2}\ \left(\e^{\gamma}+\|u_0\|_{\infty}\right)^{(2-q)/2q}\ \left( 1+\e^{q/2} \right)^{1/q}\ t^{-1/q} \\
& \le & C\ \left(\e^{\gamma}+\|u_0\|_{\infty}\right)^{1/q}\ \left( 1+\e^{q/2} \right)^{1/q}\ t^{-1/q}
\end{eqnarray*}
for any $(t,x)\in(0,\e^{(2\gamma-q)/2})\times\real^N$.

\subsubsection{$q>1$.}\label{newsubsub6}

We set $\varrho(z)=z^{1/q}$ for $z\ge 0$. Owing to \eqref{new2}, we
have
\begin{equation*}
\begin{split}
\tilde{R}_1^{\e} & \ge \frac{u_{\e}^{(2-q)/2}}{q^2}\ \left[ (p-1)(k+q-1)\ g_\e^{p-2} + \e^2\ \left( (2-p)(q-1)-k(p-1) \right)\ g_\e^{p-4} \right] \\
& \ge \e^2\ \frac{u_{\e}^{(2-q)/2}}{q^2}\ \left[ (p-1)(k+q-1) + (2-p)(q-1)-k(p-1) \right]\ g_\e^{p-4}  \\
& \ge \e^2\ \frac{ (q-1)\ u_{\e}^{(2-q)/2}}{q^2}\ g_\e^{p-4} \ge 0\,.
\end{split}
\end{equation*}
Since we are interested only in the effect of the Hamilton-Jacobi
part, we omit this term. Next, arguing as in \cite{BL99} and using
\eqref{approx.est}, we obtain
\begin{equation*}
\begin{split}
\tilde{R}_2^{\e}\ w_\e & =\frac{1}{qu_{\e}}\ \left[(q-1)g_\e^q+\e^q-q\e^2g_\e^{q-2}\right]w_\e \geq \frac{\min{\{ 1, q-1 \}}}{q u_{\e}}(g_\e^q-\e^q)\ w_\e\\
&\geq C_1\ u_{\e}^{-1}\ \varrho(u_{\e})^q\ w_\e^{(2+q)/2} - \e^q\ u_{\e}^{-1}\ w_\e \geq C_1\ w_\e^{(2+q)/2} - \e^{q-\gamma}\ w_\e\,.
\end{split}
\end{equation*}
We obtain that
\begin{equation*}
L_{\e}w_\e:=\partial_{t} w_\e - A_\e w_\e - B_\e \cdot \nabla w_\e + C_1\ w_\e^{(2+q)/2} - \e^{q-\gamma}\ w_\e \leq 0
\end{equation*}
in $Q_\infty$. Following the same computations as in Section~\ref{subsec.approx.grad}, we notice that the function
\begin{equation*}
W_\e(t):=\left(\frac{2+q \e^{q/2}}{qC_1}\right)^{2/q}\ t^{-2/q}\,, \quad t>0\,,
\end{equation*}
is a supersolution for the differential operator $L_{\e}$ in
$(0,\e^{(2\gamma-q)/2})\times\real^N$. We then infer from the
comparison principle that
\begin{equation}\label{approx.grad.HJ2}
\left|\nabla u_{\e}^{(q-1)/q}(t,x)\right|\leq\frac{q-1}{q}\ \left(\frac{2+q\e^{q/2}}{qC_1}\right)^{1/q}\ t^{-1/q}
\end{equation}
for any $(t,x)\in(0,\e^{(2\gamma-q)/2})\times\real^N$.

\subsection{Existence}

We have to pass to the limit as $\e\to 0$, and to this aim we follow
the lines of \cite[Section~3]{BtL08}. The uniform gradient bound
\eqref{approx.grad.est0} ensures that the family $u_{\e}$ is
equicontinuous with respect to the space variable and we next argue
as in \cite[Lemma~5]{GGK03} to establish the time equicontinuity. As
a consequence, we are in a position to apply the Arzel\`a-Ascoli
theorem and conclude that there exists a limit
$$
u(t,x):=\lim\limits_{\e\to0}u_{\e}(t,x),
$$
with uniform convergence in compact subsets of
$[0,\infty)\times\real^N$. By the stability result for viscosity
solutions \cite[Theorem 6.1]{OS}, we conclude that $u$ is a
viscosity solution for the equation \eqref{eq1} with initial
condition $u_0$, satisfying moreover that
\begin{equation*}
0\leq u(t,x)\leq\|u_0\|_{\infty}.
\end{equation*}
Finally, the dependence on $\e$ in the right-hand side of the
approximate gradient estimates
\eqref{approx.grad.est1}-\eqref{approx.grad.est8} (depending on the
range of the exponents $p$ and $q$) and in the time interval
validity of these estimates allow us to pass to the limit in an
uniform way, while in the left-hand side we can pass to the limit in
the gradient terms in the weak sense. We thus end the proof of the
gradient estimates in Theorems~\ref{th:grad1}, \ref{th:grad2}
and~\ref{th:grad3}. In addition, using \cite[Theorem~4.1]{BM92}, it
can be shown (as in \cite{BtL08}) that
\begin{equation*}
\nabla u_{\e}\to\nabla u \quad \hbox{a.}\ \hbox{e.}\ \hbox{in} \ Q_{\infty}\,,
\end{equation*}
so that $u$ is also a weak solution to \eqref{eq1} and satisfies
\eqref{weakform} and also \eqref{decaymass}. Finally, the uniqueness
assertion follows from \cite[Theorem~3.1]{OS}.

\section*{Acknowledgments}

RI is supported by the ANR project CBDif-Fr ANR-08-BLAN-0333-01.
Part of this work was done while PhL enjoys the hospitality and
support of the Isaac Newton Institute for Mathematical Sciences,
Cambridge, UK.

\bibliographystyle{plain}

\end{document}